\documentclass[10pt]{article}
\usepackage{eurosym}
\usepackage{amsfonts}
\usepackage{amssymb}
\usepackage{amsthm}
\usepackage{amsmath}
\usepackage{graphicx}
\usepackage{empheq}
\usepackage{indentfirst}
\usepackage{cite}
\usepackage{mathrsfs}
\usepackage{cases}
\usepackage{graphics}
\usepackage{xcolor}
\usepackage{bm}
\usepackage{makeidx}
\usepackage[T1]{fontenc}
\usepackage{times}

\newtheorem{theorem}{\textbf{Theorem}}[section]
\newtheorem{lemma}{\textbf{Lemma}}[section]
\newtheorem{proposition}{\textbf{Proposition}}[section]
\newtheorem{corollary}{\textbf{Corollary}}[section]
\newtheorem{remark}{\textbf{Remark}}[section]
\newtheorem{definition}{\textbf{Definition}}[section]

\allowdisplaybreaks[4]

\def\be{\begin{equation}}
\def\ee{\end{equation}}
\def\bea{\begin{eqnarray}}
\def\eea{\end{eqnarray}}
\def\bt{\begin{theorem}}
\def\et{\end{theorem}}
\def\bl{\begin{lemma}}
\def\el{\end{lemma}}
\def\br{\begin{remark}}
\def\er{\end{remark}}
\def\bp{\begin{proposition}}
\def\ep{\end{proposition}}
\def\bc{\begin{corollary}}
\def\ec{\end{corollary}}
\def\bd{\begin{definition}}
\def\ed{\end{definition}}

\begin{document}

\title{On the Cahn-Hilliard equation with kinetic rate dependent \\
dynamic boundary condition and non-smooth potential:\\
separation property and long-time behavior}
\author{
Maoyin Lv \thanks{%
School of Mathematical Sciences, Fudan University,
Shanghai 200433, P.R. China.
Email: \texttt{mylv22@m.fudan.edu.cn} }, \ \
Hao Wu \thanks{%
Corresponding author. School of Mathematical Sciences, Fudan University,
Shanghai 200433, P.R. China.
Email: \texttt{haowufd@fudan.edu.cn} } }
\date{\today }
\maketitle

\begin{abstract}
\noindent We consider a class of Cahn-Hilliard equation that characterizes phase separation phenomena of binary mixtures in a bounded domain $\Omega \subset \mathbb{R}^d$ $(d\in \{2,3\})$ with non-permeable boundary. The equations in the bulk are subject to kinetic rate dependent dynamic boundary conditions with  possible boundary diffusion acting on the boundary chemical potential. For the initial boundary value problem with singular potentials, we prove that any global weak solution exhibits a propagation of regularity in time. In the two dimensional case, we establish the instantaneous strict separation property by a suitable De Giorgi's iteration scheme, which yields that the weak solution stays uniformly away from the pure phases $\pm 1$ from any positive time on. In particular, when the bulk and boundary chemical potentials are in equilibrium, we obtain the instantaneous separation property with or without possible boundary diffusion acting on the boundary chemical potential. Next, in the three dimensional case, we show the eventual strict separation property that holds after a sufficiently large time. These separation properties are obtained in an unified way with respect to the structural parameters. Moreover, they allow us to achieve higher-order regularity of the global weak solution and prove the convergence to a single equilibrium as $t\rightarrow \infty$.
\medskip

\smallskip

\end{abstract}


\section{Introduction}

In this work, we consider the Cahn-Hilliard equation
\begin{align}
\left\{
    \begin{array}{ll}
       \partial_{t}\varphi=\Delta\mu,  &\quad \text{in }\Omega\times(0,\infty), \\
      \mu=-\Delta\varphi+F'(\varphi),   & \quad\text{in }\Omega\times(0,\infty),
    \end{array}\right.\label{CH}
\end{align}
subject to the so-called kinetic rate dependent dynamic boundary conditions
\begin{align}
    \left\{
    \begin{array}{ll}
\partial_{t}\varphi=\sigma\Delta_{\Gamma}\theta-\partial_{\mathbf{n}}\mu-\alpha\theta,
       & \quad\text{on }\Gamma\times(0,\infty), \\   \theta=\partial_{\mathbf{n}}\varphi-\nu\Delta_{\Gamma}\varphi+G'(\varphi),  & \quad\text{on }\Gamma\times(0,\infty),\\
       L\partial_{\mathbf{n}}\mu=\theta-\mu,&\quad\text{on }\Gamma\times(0,\infty),
    \end{array}\right.\label{dynamic}
\end{align}
and the initial condition
\begin{align}
\varphi(0)=\varphi_{0}\quad\text{in }\overline{\Omega}.\label{Chini}
\end{align}
Here, $\Omega\subset\mathbb{R}^{d}$ $(d\in\{2,3\})$ is a bounded domain with smooth boundary $\Gamma:=\partial\Omega$. The bold letter $\mathbf{n}=\mathbf{n}(x)$ denotes the unit outer normal vector on $\Gamma$ and $\partial_{\mathbf{n}}$ is the outward normal derivative on the boundary. The symbols $\Delta$, $\Delta_{\Gamma}$ denote the usual Laplace operator in $\Omega$ and the Laplace-Beltrami operator on $\Gamma$, respectively.

The Cahn-Hilliard equation \eqref{CH}, first introduced in \cite{CH}, serves as a fundamental diffuse-interface model that describes the phase separation process of binary mixtures. The unknown $\varphi:\overline{\Omega}\times(0,\infty)\rightarrow [-1,1]$ is a phase function that represents the difference between local concentrations of the two components of a binary mixture (often called the order parameter of the system) and the function $\mu:\Omega\times(0,\infty)\rightarrow\mathbb{R}$ stands for the (bulk) chemical potential. When the evolution described by \eqref{CH} is confined in a bounded domain $\Omega$, suitable boundary conditions should be taken into account. Classical choices are the homogeneous Neumann boundary conditions
\begin{align}
    \partial_{\mathbf{n}}\varphi=\partial_{\mathbf{n}}\mu=0\quad\text{on }\Gamma\times(0,\infty).\notag
\end{align}
The resulting initial boundary value problem has been analyzed in the literature from different viewpoints (cf. \cite{AW,GGM,KNP,K,RH,W} and the references therein).

In recent years, boundary effects in the phase separation process of binary mixtures have attracted a lot of attention, for instance, the effective short-range interactions between a polymer blend and the solid wall \cite{FMD,KEMRSBD}, and the contact line problem in hydrodynamic applications \cite{QWS}. To this end, physicists suggest the total free energy consisting of a Ginzburg-Landau type energy in the bulk and a surface contribution on the boundary (see \cite{FMD,KEMRSBD}):
\begin{align}
	E\big(\varphi\big):=\underbrace{\int_{\Omega}\Big(\frac{1}{2}|\nabla \varphi|^{2}+F(\varphi)\Big)\,\mathrm{d}x}_{\text{bulk free energy}}+\underbrace{\int_{\Gamma}\Big(\frac{\nu}{2}|\nabla_{\Gamma}\varphi|^{2}+G(\varphi)\Big)\,\mathrm{d}S}_{\text{surface free energy}}.\label{energy}
\end{align}
In \eqref{energy}, the symbols $\nabla$ and $\nabla_{\Gamma}$ denote the usual gradient operator and the tangential (surface) gradient operator. The parameter $\nu\in[0,\infty)$ represents possible surface diffusion.
Comparing with $\mu$, the function $\theta:\Gamma\times (0,\infty) \rightarrow\mathbb{R}$ in \eqref{dynamic} stands for the chemical potential on the boundary. The bulk and boundary chemical potentials $\mu$, $\theta$ are related to Fr\'echet derivatives of the bulk and surface free energy, respectively. They provide a measure of the change in bulk/surface energy with respect to small perturbation of the order parameter.
The nonlinear potential functions $F$ and $G$  in \eqref{energy} are energy densities in the bulk and on the boundary. A thermodynamically relevant example is given by the so-called logarithmic potential \cite{CH}:
\begin{align}	
F_{\mathrm{log}}(r):=\frac{\Theta}{2}\big[(1+r)\mathrm{ln}(1+r)+(1-r)\mathrm{ln}(1-r)\big] -\frac{\Theta_c}{2}r^{2},\quad r\in(-1,1),
\label{logari}
\end{align}
where $\Theta>0$ is the absolute temperature and $\Theta_c>0$ is the critical temperature. The
parameter $\Theta_c$ gives a measure of the strength of the de-mixing effect.
When $\Theta_c>\Theta$, $F_{\mathrm{log}}$ is nonconvex with a double-well structure and the phase separation occurs. Singular potentials like \eqref{logari} are important since they guarantee the solution $\varphi$ to take value in the physical interval $[-1, 1]$.

Based on the total free energy \eqref{energy}, various boundary conditions for the Cahn-Hilliard equation \eqref{CH} have been discussed in the literature (see the recent review article \cite{W}). Here, we are interested in a class of dynamic boundary conditions given by \eqref{dynamic}. The bulk and boundary chemical potentials $\mu$, $\theta$ are coupled through the boundary condition \eqref{dynamic}$_{3}$, where the parameter $L\in[0,\infty]$ is a non-negative constant. Its reciprocal $1/L$ is related to a kinetic rate and the term $L\partial_{\mathbf{n}}\mu$ describes adsorption and desorption processes between the materials in the bulk and on the boundary \cite{KLLM}. Besides, the parameter $\sigma \in[0,\infty)$ represents possible surface diffusion acting on the surface chemical potential and $\alpha\in[0,\infty)$ is related to possible mass exchange to the environment \cite{Gal06}.

The initial boundary value problem \eqref{CH}--\eqref{Chini} satisfies two important properties, that is, energy dissipation and mass conservation \cite{Gal06,Gal,GMS,KLLM,LW}. For sufficiently regular solutions, it is straightforward to check that
\begin{align}
    \frac{\mathrm{d}}{\mathrm{d}t}E\big(\varphi\big)+\int_{\Omega}|\nabla\mu|^{2}\,\mathrm{d}x +\int_{\Gamma}\Big(\sigma|\nabla_{\Gamma}\theta|^{2}+\alpha|\theta|^{2}\Big)\,\mathrm{d}S +\chi(L)\int_{\Gamma}(\theta-\mu)^{2}\,\mathrm{d}S=0,\quad \forall\, t\in (0,\infty),
    \label{formalener}
\end{align}
with
\begin{align*}
    \chi(L)=\left\{
    \begin{array}{ll}
       1/L,  &\text{if }L\in(0,\infty),  \\
        0, &\text{if }L=0\text{ or }\infty.
    \end{array}\right.
\end{align*}
The basic energy law \eqref{formalener} implies that different values of $L$, $\sigma$, $\alpha$ can yield different types of energy dissipation. On the other hand, we can derive the mass balance relation
\begin{align*}
    \frac{\mathrm{d}}{\mathrm{d}t}\Big(\int_{\Omega}\varphi\,\mathrm{d}x +\int_{\Gamma}\varphi\,\mathrm{d}S\Big)=-\alpha\int_{\Gamma}\theta\,\mathrm{d}S,\quad\forall\, t\in (0,\infty).
\end{align*}
The parameter $\alpha$ distinguishes the cases of permeable wall $(\alpha>0)$ and non-permeable wall $(\alpha=0)$, and when $\alpha=0$, the total mass is conserved:
\begin{align}
\int_{\Omega}\varphi(t)\,\mathrm{d}x+\int_{\Gamma}\varphi(t)\,\mathrm{d}S =\int_{\Omega}\varphi_{0}\,\mathrm{d}x+\int_{\Gamma}\varphi_{0}\,\mathrm{d}S,\quad \forall\,t\in[0,\infty).
\label{totalmass}
\end{align}
Furthermore, if $L=\infty$, the boundary condition \eqref{dynamic}$_{3}$ reduces to $\partial_{\mathbf{n}}\mu=0$ on $\Gamma\times(0,\infty)$, then instead of \eqref{totalmass}, the bulk and boundary mass are conserved separately (see \cite{LW}):
\begin{align}
\int_{\Omega}\varphi(t)\,\mathrm{d}x=\int_{\Omega}\varphi_{0}\,\mathrm{d}x,\quad\int_{\Gamma}\varphi(t)\,\mathrm{d}S=\int_{\Gamma}\varphi_{0}\,\mathrm{d}S,\quad t\in[0,\infty).\notag
\end{align}

Let us briefly mention some related works in the literature.
The case $L=0$ implies that $\mu$ and $\theta$ are in the chemical equilibrium, that is,
$\theta=\mu|_{\Gamma}$ on $\Gamma\times(0,\infty)$. When $L=0$, $\sigma=0$, $\nu>0$ and $\alpha>0$, such a model was first proposed in \cite{Gal06} based on the mass balance and variational method. The corresponding boundary condition is also referred to as the Wentzell boundary condition. Later in \cite{GMS}, an extended model with $L=0$, $\sigma\geq0$, $\nu\geq0$ and $\alpha=0$ was derived. For mathematical analysis including well-posedness and long-time behavior, we refer to \cite{Gal06,Ka,Gal,Wu,GW} concerning regular potentials and \cite{CGM,CF15,GMS} concerning singular potentials. In the case with non-permeable wall ($\alpha=0$) and $\sigma, \nu \geq 0$, well-posedness, regularity and long-time behavior of global weak solutions were established in \cite{GMS} under general assumptions on the potentials, with $F$ being singular and $G$ being regular (see also \cite{CGM} for further results with $\sigma,\nu>0$ in a proper variational formulation). If both the bulk and boundary potentials are singular, results on well-posedness, including the existence and uniqueness of global weak solutions and strong solutions were proven in \cite{CF15}. Later in \cite{FW}, the authors proved the strict separation property and convergence to equilibrium as $t\to \infty$. Well-posedness in the absence of the surface diffusion acting on the boundary phase variable was obtained in \cite{CFS} via the asymptotic limit as $\nu\rightarrow0$. We also refer to \cite{CGS18,GS} when the effects of convection and viscous regularization were taken into account.

The case $L=\infty$ was derived in \cite{LW} by an energetic variational approach that combines the least action principle and Onsager's principle of maximum energy dissipation. In this situation, the bulk and boundary chemical potentials are not directly coupled, while the interaction between bulk and boundary materials is represented  through the trace condition of the order parameter $\varphi$. When $\nu\geq 0$, $\sigma>0$ and $\alpha=0$, well-posedness and long-time behavior of the model with regular potentials were addressed in \cite{LW}. Later in  \cite{GK}, the authors proved the existence of weak solutions in a weaker form by the gradient flow approach, removing the additional geometric assumption imposed in \cite{LW} when $\nu=0$. Well-posedness for the model with general singular potentials was recently obtained in \cite{CFW}. For long-time behavior related to the existence of global attractors, we refer to \cite{MW}. Recently in \cite{CFSJEE}, the asymptotic limit of vanishing boundary diffusion as $\nu\rightarrow0$ was analyzed.

The intermediate case $L\in (0,\infty)$ was proposed in \cite{KLLM}, which provides a connection between the above mentioned models with $L=0$ and $L=\infty$. For regular potentials under suitable growth conditions, the authors of \cite{KLLM} proved  well-posedness by the gradient flow approach and studied the asymptotic limits as $L\rightarrow0$ and $L\rightarrow\infty$, respectively. We refer to \cite{GKY} for long-time behavior of global weak solutions including the existence of global/exponential attractors and convergence to a single equilibrium. An extended model with convection was recently analyzed in \cite{KS24}. We also mention \cite{KS} in which a nonlocal variant with $L\in [0,\infty]$ and regular potentials was investigated. However, the case with singular potentials $F$, $G$ turns out to be more involved. Singular potentials like \eqref{logari} are not admissible in the above mentioned works. To the best of our knowledge, the only known contribution in this direction is \cite{LvWu}, in which existence and uniqueness of global weak as well as strong solutions were obtained for a general class of non-smooth potentials. In the same paper, the authors studied the asymptotic limit as $\nu\rightarrow0$, which leads to the well-posedness without surface diffusion acting on boundary phase variable. Moreover, relations between all cases with $L=0$, $L\in(0,\infty)$ and $L=\infty$ were established via the asymptotic limits as $L\rightarrow0$ and $L\rightarrow\infty$.

In this work, we continue our study in \cite{LvWu} with singular bulk and boundary potential functions. For this purpose, we focus on the problem \eqref{CH}--\eqref{Chini} in a bounded domain with non-permeable wall (i.e., $\alpha=0$) and take into account boundary diffusion acting on the boundary phase variable (i.e., $\nu>0$). Without loss of generality, let us set $\nu=1$. Then we maintain two parameters $(L,\sigma)\in[0,\infty)\times[0,\infty)$ that are important in the subsequent analysis. More precisely, the following combination of parameters will be treated:
\begin{align}
    \begin{array}{lllll}
       \bullet\  & L\in(0,\infty),& \sigma\in(0,\infty), &\nu=1,&\alpha=0; \\
       \bullet\  & L=0, & \sigma\in[0,\infty), &\nu=1,&\alpha=0.
    \end{array} \label{choice}
\end{align}
When $L=0$, we are allowed to take $\sigma=0$, since the corresponding Dirichlet boundary condition $$\mu|_{\Gamma}=\theta\quad\text{on }\Sigma$$
together with the regularity of $\mu
\in H^1(\Omega)$ and the trace theorem yields that $\theta\in H^{1/2}(\Gamma)$. This fact provides crucial estimates in the corresponding analysis for boundary terms. The case $L=\infty$ will be treated separately in a future study, because a different approximating scheme has to be adopted (cf. \cite{CFW}).

Our aim is to study the regularity properties for global weak solutions to problem \eqref{CH}--\eqref{Chini} and to characterize their long-time behavior. The results are summarized as follows:

\begin{itemize}
\item[(1)] Regularity of global weak solutions. Under suitable assumptions on the singular potentials, we first prove that any global weak solution exhibits a propagation of regularity in time (see Theorem \ref{eventual}). Then we show the strict separation property provided that the initial datum is not a pure state: in both two and three dimensions, the eventual strict separation will happen after a sufficiently large time $T_{\mathrm{SP}}$ (see Theorem \ref{eventual-2}); while in two dimensions, the global weak solution will stay uniformly away from the pure states $\pm1$ instantaneously for positive time (see Theorem \ref{intantaneous}).
\item[(2)] Convergence to a single equilibrium. After establishing the strict separation property of global weak solutions, the singular potentials can be regarded as globally Lipschitz nonlinearities from a certain time on. Then under the additional assumption that $F$, $G$ are real analytic, we are able to prove that every global weak solution converges to a single equilibrium as $t\rightarrow\infty$ by the {\L}ojasiewicz-Simon approach (see Theorem \ref{equilibrium}).
\end{itemize}

We emphasize that the strict separation property for the order parameter plays a fundamental role in the study of phase-field models with singular potentials. For the Cahn-Hilliard equation \eqref{CH} subject to homogeneous Neumann boundary conditions, the eventual separation property was proven in \cite{AW} in  both two and three dimensions. The proof relies on the energy dissipation structure of the Cahn-Hilliard equation, which drives the evolution towards the set of steady states that are separated from the pure phases, see also \cite{GP} for an alternatively proof based on De Giorgi's iteration scheme.
On the other hand, the instantaneous separation property turns out to be more involved. Its validity under the logarithmic potential \eqref{logari} in two dimensions was first proven in \cite{MZ04}, while the three dimensional case is less satisfactory, since stronger singularity of the potential near $\pm 1$ is required, see \cite{LP18,MZ04}. A different approach for the instantaneous separation in two dimensions was given in \cite{GGM} for the Cahn-Hilliard-Oono equation. The argument therein was further simplified in \cite{HW21} and later in \cite{GGG}, a more direct argument can be found for a wider class of singular potentials. The methods in \cite{GGM,GGG,HW21,MZ04} rely on some specific relation between $\widehat{\beta}'$ and $\widehat{\beta}''$ ($\widehat{\beta}$ denotes the non-smooth convex part of $F$ with decomposition $F=\widehat{\beta}+\widehat{\pi}$, where $\widehat{\pi}$ is a smooth concave perturbation) such that
\begin{align}
	\widehat{\beta}''(s)\leq C_{\sharp}e^{C_{\sharp}|\widehat{\beta}'(s)|^{\kappa_{\sharp}}},
\quad\forall\,s\in(-1,1),\label{second}
\end{align}
for some constants $C_{\sharp}>0$ and $\kappa_{\sharp}\in[1,2)$ (in \cite{GGM,HW21,MZ04}, only the case $\kappa_{\sharp}=1$ was treated). Recently, with the aid of a suitable De Giorgi's iteration scheme, the authors of \cite{GP} were able to handle a more general set of assumptions on the singular potential, that is, as $\delta\rightarrow0$, there exists some $\kappa>1/2$, such that
\begin{align*}
	\frac{1}{\widehat{\beta}'(1-2\delta)} =O\Big(\frac{1}{|\ln \delta|^{\kappa}}\Big), \quad\frac{1}{|\widehat{\beta}'(-1+2\delta)|}=O\Big(\frac{1}{|\ln \delta|^{\kappa}}\Big).
\end{align*}
This yields a milder growth condition of the first derivative $\widehat{\beta}'$ near the pure states $\pm1$, but without any pointwise additional assumption on the second derivative $\widehat{\beta}''$ (cf. \eqref{second}). For recent developments concerning the nonlocal and fractional Cahn-Hilliard equations, we refer to \cite{GGG,Gi,GP,Po} and the references therein.

Returning to problem \eqref{CH}--\eqref{Chini}, when $L=\alpha=0$, $\sigma\geq0$, $\nu>0$, the authors of \cite{FW} proved the eventual separation property in both two and three dimensions using the dynamic approach as in \cite{AW}. Besides, in the two dimensional case, under the additional assumptions $\sigma>0$ and  \eqref{second} with $\kappa_{\sharp}=1$, they established the instantaneous separation property. In this study, we extend the argument in \cite{GP} to a wider class of problems with different types of bulk-surface interactions (recall \eqref{choice}). Under suitable compatibility condition between the bulk and boundary potentials, we are able to apply a modified De Giorgi's iteration scheme to establish the instantaneous separation property in two dimensions and the eventual separation property in two and three dimensions. Our argument is unified with respect to different choices of parameters described in \eqref{choice}.
Besides, our results extend the previous work \cite{FW} for $L = 0$ under a weaker assumption on the singular potential and without the regularizing effect due to surface diffusion $\sigma>0$ (see Remark \ref{rem-2d}).

The rest of this paper is organized as follows:
in Section 2, we introduce the functional settings and basic assumptions, after that we state the main results of this study. In Section 3, we show the propagation of regularity for global weak solutions for positive time.
Section 4 presents a comprehensive analysis of the eventual as well as instantaneous separation properties. Finally, the convergence to a single equilibrium is proven in Section 5. In the Appendix, we list some mathematical tools that are used in this paper.

\section{Main results}
\setcounter{equation}{0}
In this section, we first recall some notations for function spaces. After that we present the basic assumptions and state the main results.

\subsection{Preliminaries and basic settings}
For any real Banach space $X$, we denote its norm by $\|\cdot\|_X$, its dual space by $X'$
and the duality pairing between $X'$ and $X$ by
$\langle\cdot,\cdot\rangle_{X',X}$. If $X$ is a Hilbert space,
we denote its associated inner product by $(\cdot,\cdot)_X$.
The space $L^q(0,T;X)$ ($1\leq q\leq \infty$)
stands for the set of all strongly measurable $q$-integrable functions with
values in $X$, or, if $q=\infty$, essentially bounded functions. The space $L^p_{\mathrm{uloc}}(0,+\infty;X)$ denotes the uniformly local variant of $L^p(0,+\infty;X)$ consisting of all strongly measurable $f:[0,+\infty)\to X$ such that
\begin{equation*}
    \|f\|_{L^p_{\mathrm{uloc}}(0,+\infty;X)}:= \mathop\mathrm{sup}\limits_{t\ge0} \|f\|_{L^p(t,t+1;X)} < \infty.
\end{equation*}
If $T\in(0,+\infty)$, we find $L^p_{\mathrm{uloc}}(0,T;X)=L^p(0,T;X)$.
The space $C([0,T];X)$ denotes the Banach space of all bounded and
continuous functions $u:[ 0,T] \rightarrow X$ equipped with the supremum
norm, while $C_{w}([0,T];X)$ denotes the topological vector space of all
bounded and weakly continuous functions.

Let $\Omega$ be a bounded domain in $\mathbb{R}^d$ ($d\in \{2,3\}$) with sufficiently smooth boundary $\Gamma:=\partial \Omega$.
We denote by $|\Omega|$ and $|\Gamma|$
the Lebesgue measure of $\Omega$ and the Hausdorff measure of $\Gamma$, respectively.
For any $1\leq q\leq \infty$, $k\in \mathbb{N}$, the standard Lebesgue and Sobolev spaces on $\Omega$ are denoted by $L^{q}(\Omega )$
and $W^{k,q}(\Omega)$. Here, we use $\mathbb{N}$ for the set of natural numbers including zero.
For $s\geq 0$ and $q\in [1,\infty )$, we denote by $H^{s,q}(\Omega )$ the Bessel-potential spaces and by $W^{s,q}(\Omega )$ the Slobodeckij spaces.
If $q=2$, it holds $H^{s,2}(\Omega)=W^{s,2}(\Omega )$ for all $s$ and these spaces are Hilbert spaces.
We use the notations $H^s(\Omega)=H^{s,2}(\Omega)=W^{s,2}(\Omega )$ and $H^0(\Omega)$ can be identified with $L^2(\Omega)$.
The Lebesgue spaces, Sobolev spaces and Slobodeckij spaces on the boundary $\Gamma$ can be defined analogously,
provided that $\Gamma$ is sufficiently regular.
Again, we write $H^s(\Gamma)=H^{s,2}(\Gamma)=W^{s,2}(\Gamma)$ and identify $H^0(\Gamma)$ with $L^2(\Gamma)$.
For convenience, we apply the following shortcuts:
  \begin{align*}
  &H:=L^{2}(\Omega),\quad H_{\Gamma}:=L^{2}(\Gamma),\quad V:=H^{1}(\Omega),\quad V_{\Gamma}:=H^{1}(\Gamma).
  \end{align*}

Next, we introduce the product spaces
$$\mathcal{L}^{q}:=L^{q}(\Omega)\times L^{q}(\Gamma)\quad\mathrm{and}\quad\mathcal{H}^{k}:=H^{k}(\Omega)\times H^{k}(\Gamma),$$
for $q\in [1,\infty]$ and $k\in \mathbb{R}$, $k\geq 0$.
Like before, we can identify  $\mathcal{H}^{0}$ with $\mathcal{L}^{2}$.
 For any $k\in \mathbb{N}$, $\mathcal{H}^{k}$ is a Hilbert space endowed with the standard inner product
  $$
  \big((y,y_{\Gamma}),(z,z_{\Gamma})\big)_{\mathcal{H}^{k}}:=(y,z)_{H^{k}(\Omega)}+(y_{\Gamma},z_{\Gamma})_{H^{k}(\Gamma)},\quad\forall\, (y,y_{\Gamma}), (z,z_{\Gamma})\in\mathcal{H}^{k}
  $$
  and the induced norm $\Vert\cdot\Vert_{\mathcal{H}^{k}}:=(\cdot,\cdot)_{\mathcal{H}^{k}}^{1/2}$.
  We introduce the duality pairing
  \begin{align*}
  \big\langle (y,y_\Gamma),(\zeta, \zeta_\Gamma)\big\rangle_{(\mathcal{H}^1)',\mathcal{H}^1}
  = (y,\zeta)_{L^2(\Omega)}+ (y_\Gamma, \zeta_\Gamma)_{L^2(\Gamma)},
  \quad \forall\, (y,y_\Gamma)\in \mathcal{L}^2,\ (\zeta, \zeta_\Gamma)\in \mathcal{H}^1.
  \end{align*}
  By the Riesz representation theorem, this product
can be extended to a duality pairing on $(\mathcal{H}^1)'\times \mathcal{H}^1$.
For any $k\in\mathbb{Z}^+$, we introduce the Hilbert space
  $$
  \mathcal{V}^{k}:=\big\{(y,y_{\Gamma})\in\mathcal{H}^{k}\;:\;y|_{\Gamma}=y_{\Gamma}\ \ \text{a.e. on }\Gamma\big\},
  $$
  endowed with the inner product $(\cdot,\cdot)_{\mathcal{V}^{k}}:=(\cdot,\cdot)_{\mathcal{H}^{k}}$ and the associated norm $\Vert\cdot\Vert_{\mathcal{V}^{k}}:=\Vert\cdot\Vert_{\mathcal{H}^{k}}$.
  Here, $y|_{\Gamma}$ stands for the trace of $y\in H^k(\Omega)$ on the boundary $\Gamma$, which makes sense for $k\in \mathbb{Z}^+$.
  The duality pairing on $(\mathcal{V}^1)'\times \mathcal{V}^1$ can be defined in a similar manner.
  For convenience, we also introduce the notation
  $$
  \widetilde{\mathcal{V}}^{k}:=\big\{(y,y_{\Gamma})\in H^{k}(\Omega)\times H^{k-1/2}(\Gamma)\;:\;y|_{\Gamma}=y_{\Gamma}\ \ \text{a.e. on }\Gamma\big\},
  \quad  k\in \mathbb{Z}^+.
  $$
  Thanks to the trace theorem, for every $y\in H^k(\Omega)$, $k\in \mathbb{Z}^+$, it holds $(y,y|_{\Gamma})\in  \widetilde{\mathcal{V}}^{k}$. Besides, we consider the duality pairing between $(\widetilde{\mathcal{V}}^{1})'$ and $\widetilde{\mathcal{V}}^{1}$ as a continuous extension of the inner
product on $\mathcal{L}^2$, that is
 \begin{align*}
  \big\langle (y,y_\Gamma),(\zeta, \zeta_\Gamma)\big\rangle_{(\widetilde{\mathcal{V}}^{1})',\widetilde{\mathcal{V}}^{1}}
  = (y,\zeta)_{L^2(\Omega)}+ (y_\Gamma, \zeta_\Gamma)_{L^2(\Gamma)},
  \quad \forall\, (y,y_\Gamma)\in \mathcal{L}^2,\ (\zeta, \zeta_\Gamma)\in \widetilde{\mathcal{V}}^{1}.
  \end{align*}

For every $y\in (H^1(\Omega))'$, we denote by $%
\langle y\rangle_\Omega=|\Omega|^{-1}\langle
y,1\rangle_{(H^1(\Omega))',\,H^1(\Omega)}$ its generalized mean
value over $\Omega$. If $y\in L^1(\Omega)$, then its spatial mean is simply
given by $\langle y\rangle_\Omega=|\Omega|^{-1}\int_\Omega y \,\mathrm{d}x$.
The spatial mean for a function $y_\Gamma$ on $\Gamma$, denoted by $\langle
y_\Gamma\rangle_\Gamma$, can be defined in a similar manner.
For any given $m\in\mathbb{R}$, we set
 \begin{align}
 \mathcal{L}^{2}_{(m)}:=\big\{(y,y_{\Gamma})\in\mathcal{L}^{2}\;:\;\overline{m}(y,y_{\Gamma})=m\big\},
 \notag
\end{align}
 where the generalized mean is defined as
\begin{align}
\overline{m}(y,y_{\Gamma}):=\frac{|\Omega|\langle y\rangle_{\Omega}+|\Gamma|\langle y_{\Gamma}\rangle_{\Gamma}}{|\Omega|+|\Gamma|}.\label{gmean}
\end{align}
 Then the closed linear subspaces
 $$
 \mathcal{H}_{(0)}^k=\mathcal{H}^{k}\cap\mathcal{L}_{(0)}^{2},
 \quad \mathcal{V}_{(0)}^k=\mathcal{V}^{k}\cap\mathcal{L}_{(0)}^{2},
 \quad k\in \mathbb{Z}^+,
  $$
 are Hilbert spaces endowed with the inner products $(\cdot,\cdot)_{\mathcal{H}^{k}}$
 and the associated norms $\Vert\cdot\Vert_{\mathcal{H}^{k}}$, respectively.

 Let us first consider the case with parameters $L\in [0,\infty)$, $\sigma\in(0,\infty)$. We introduce the notations
  $$
  \mathcal{H}^{k}_{L,\sigma}:=
  \begin{cases}
  \mathcal{H}^k,\quad \text{if}\ L\in (0,\infty),\\
  \mathcal{V}^{k},\quad \ \text{if}\ L=0,
  \end{cases}\qquad
  \mathcal{H}^{k}_{L,\sigma,0}:=\mathcal{H}_{L,\sigma}^{k}\cap\mathcal{L}_{(0)}^{2}=
  \begin{cases}
  \mathcal{H}_{(0)}^k,\quad \text{if}\ L\in (0,\infty),\\
  \mathcal{V}^{k}_{(0)},\quad \ \text{if}\ L=0,
  \end{cases}
  $$
  for $k\in \mathbb{Z}^+$.
  Define the bilinear form
  \begin{align}
  a_{L,\sigma}\big((y,y_{\Gamma}),(z,z_{\Gamma})\big) :=\int_{\Omega}\nabla y\cdot\nabla z \,\mathrm{d}x + \sigma \int_{\Gamma}\nabla_{\Gamma}y_{\Gamma}\cdot\nabla_{\Gamma}z_{\Gamma}\,\mathrm{d}S
  +\chi(L)\int_{\Gamma}(y-y_{\Gamma})(z-z_{\Gamma})\,\mathrm{d}S,
  \notag
  \end{align}
  for all $(y,y_{\Gamma}),(z,z_{\Gamma})\in \mathcal{H}^{1}$, where
  $$
  \chi(L)=\begin{cases}
  1/L,\quad \text{if}\ L\in (0,\infty),\\
  0,\qquad \ \text{if}\ L=0.
  \end{cases}
  $$
 Set
  \begin{align}
  \Vert(y,y_{\Gamma})\Vert_{\mathcal{H}^{1}_{L,\sigma,0}}
  :=\big((y,y_{\Gamma}),(y,y_{\Gamma})\big)_{\mathcal{H}^{1}_{L,\sigma,0}}^{1/2}
  = \big[ a_{L,\sigma}\big((y,y_{\Gamma}),(y,y_{\Gamma})\big)\big]^{1/2},\qquad \forall\,
  (y,y_{\Gamma})\in \mathcal{H}^{1}_{L,\sigma,0}.
  \label{norm-hL}
  \end{align}
  We note that for $(y,y_{\Gamma})\in \mathcal{V}_{(0)}^1\subseteq\mathcal{H}^{1}_{L,\sigma,0}$, the norm $\Vert(y,y_{\Gamma})\Vert_{\mathcal{H}^{1}_{L,\sigma,0}}$ does not depend on $L$, since the third term in $a_{L,\sigma}$ simply vanishes. From the generalized Poincar\'{e}'s inequality in Lemma \ref{equivalent}, we find that $\mathcal{H}^{1}_{L,\sigma,0}$ is a Hilbert space with the inner product $(\cdot,\cdot)_{\mathcal{H}^{1}_{L,\sigma,0}}^{1/2}$. The induced norm $\Vert\cdot\Vert_{\mathcal{H}^{1}_{L,\sigma,0}}$ prescribed in \eqref{norm-hL} is equivalent to the standard one $\|\cdot\|_{\mathcal{H}^1}$ on $\mathcal{H}^{1}_{L,\sigma,0}$.
In particular, the space $\mathcal{V}_{(0)}^{1}$ has the equivalent norm given by
$$
\Vert(y,y_{\Gamma})\Vert_{\mathcal{V}_{(0)}^{1}}^{2}:=\int_{\Omega}|\nabla y|^{2}\,\mathrm{d}x +\int_{\Gamma}|\nabla_{\Gamma}y_{\Gamma}|^{2}\,\mathrm{d}S,\quad\forall\, (y,y_{\Gamma})\in\mathcal{V}_{(0)}^{1}.
$$
 For any fixed $L\in[0,\infty)$, $\sigma\in(0,\infty)$, we consider the following elliptic boundary value problem
      \begin{align}
      \left\{
      \begin{array}{rl}
      -\Delta u = y&\quad \text{in }\Omega,\\
      -\sigma\Delta_{\Gamma}u_\Gamma +\partial_{\mathbf{n}}u= y_{\Gamma}&\quad \text{on }\Gamma,\\
      L\partial_{\mathbf{n}}u =u_\Gamma-u|_{\Gamma} &\quad \mathrm{on\;}\Gamma.
      \end{array}\right.
      \label{2.2}
      \end{align}
   Define the space
   $$
   \mathcal{H}_{L,\sigma,0}^{-1}=
   \begin{cases}
   \mathcal{H}_{(0)}^{-1} :=\big\{(y,y_{\Gamma})\in(\mathcal{H}^{1})'\;:\;\overline{m}(y,y_{\Gamma})=0\big\},\quad \text{if}\ L\in(0,\infty),\\
   \mathcal{V}_{(0)}^{-1} :=\big\{(y,y_{\Gamma})\in(\mathcal{V}^{1})'\;:\;\overline{m}(y,y_{\Gamma})=0\big\},\quad \ \,\text{if}\ L=0,
   \end{cases}
   $$
   where $\overline{m}$ is given by \eqref{gmean} if $L\in (0,\infty)$, and for $L=0$, we take
   $$
   \overline{m}(y,y_{\Gamma})=\frac{\langle (y,y_{\Gamma}),(1,1)\rangle_{(\mathcal{V}^{1})', \mathcal{V}^{1}}}{|\Omega|+|\Gamma|}.
   $$
   Then the chain of inclusions holds
   $$
   \mathcal{H}^{1}_{L,\sigma,0} \subset \mathcal{L}_{(0)}^2\subset \mathcal{H}_{L,\sigma,0}^{-1} \subset (\mathcal{H}_{L,\sigma}^{1})'\subset (\mathcal{H}^{1}_{L,\sigma,0})'.
   $$
   It has been shown that (see \cite[Theorem 3.3]{KL} when $L>0$  and \cite{CF15} when $L=0$) for every $(y,y_{\Gamma})\in\mathcal{H}_{L,\sigma,0}^{-1}$, problem \eqref{2.2} admits a unique weak solution $(u,u_\Gamma)\in\mathcal{H}_{L,\sigma,0}^{1}$ satisfying the weak formulation
      \begin{align}
      a_{L,\sigma}\big((u,u_{\Gamma}),(\zeta,\zeta_{\Gamma})\big) = \big\langle(y,y_{\Gamma}),(\zeta,\zeta_{\Gamma})\big\rangle_{(\mathcal{H}^{1}_{L,\sigma})',\mathcal{H}^{1}_{L,\sigma}},
      \quad \forall\, (\zeta,\zeta_{\Gamma})\in\mathcal{H}_{L,\sigma}^{1},
      \notag
      \end{align}
      and the $\mathcal{H}^1$-estimate
      \begin{align}
      \|(u,u_{\Gamma})\|_{\mathcal{H}^1}\leq C\|(y,y_{\Gamma})\|_{(\mathcal{H}_{L,\sigma}^{1})'},\notag
      \end{align}
      for some constant $C>0$ depending only on $L$, $\sigma$ and $\Omega$. Furthermore, if $(y,y_{\Gamma})\in \mathcal{L}_{(0)}^{2}$, then it holds
      \begin{align}
      \|(u,u_{\Gamma})\|_{\mathcal{H}^{2}}\leq C\|(y,y_{\Gamma})\|_{\mathcal{L}^{2}}.
      \notag
      \end{align}
      The above facts enable us to define the solution operator $$\mathfrak{S}^{L,\sigma}:\mathcal{H}_{L,\sigma,0}^{-1}\rightarrow\mathcal{H}_{L,\sigma,0}^{1},\quad(y,y_{\Gamma})\mapsto (u,u_\Gamma)=\mathfrak{S}^{L,\sigma}(y,y_{\Gamma})=\big(\mathfrak{S}^{L,\sigma}_{\Omega}(y,y_{\Gamma}),\mathfrak{S}^{L,\sigma}_{\Gamma}(y,y_{\Gamma})\big).
      $$
      A direct calculation yields that
      $$
      \big((u,u_{\Gamma}), (z,z_\Gamma)\big)_{\mathcal{L}^2}
      =\big((u,u_{\Gamma}), \mathfrak{S}^{L,\sigma}(z,z_\Gamma)\big)_{\mathcal{H}^1_{L,\sigma,0}},\quad \forall\, (u,u_{\Gamma})\in \mathcal{H}_{L,\sigma,0}^1,\ (z,z_\Gamma)\in \mathcal{L}^2_{(0)}.
      $$
      Thanks to \cite[Corollary 3.5]{KL}, we can introduce the following inner product on $\mathcal{H}_{L,\sigma,0}^{-1}$:
      \begin{align}
      \big((y,y_{\Gamma}),(z,z_{\Gamma})\big)_{0,*}&:=
      a_{L,\sigma}\big(\mathfrak{S}^{L,\sigma}(y,y_{\Gamma}), \mathfrak{S}^{L,\sigma}(z,z_{\Gamma})\big),
      \quad \forall\, (y,y_{\Gamma}),(z,z_{\Gamma})\in \mathcal{H}_{L,\sigma,0}^{-1}.\notag
      \end{align}
      The associated norm $\Vert(y,y_{\Gamma})\Vert_{0,*} :=\big((y,y_{\Gamma}),(y,y_{\Gamma})\big)_{0,*}^{1/2}$
       is equivalent to the standard dual norm $\|\cdot\|_{(\mathcal{H}_{L,\sigma}^1)'}$ on $\mathcal{H}_{L,\sigma,0}^{-1}$.
      Then for any $(y,y_{\Gamma})\in (\mathcal{H}_{L,\sigma}^{1})'$, we have
      \begin{align}
      \|(y,y_{\Gamma})\|_{*}&:=\left(\Vert(y,y_{\Gamma})-\overline{m}(y,y_{\Gamma}) \mathbf{1}\Vert_{0,*}^2+ |\overline{m}(y,y_{\Gamma})|^2\right)^{1/2},
      \quad \forall\, (y,y_{\Gamma})\in (\mathcal{H}_{L,\sigma}^{1})',\notag
      \end{align}
      is equivalent to the usual dual norm $\|\cdot\|_{(\mathcal{H}_{L,\sigma}^{1})'}$.

Next, we consider the specific case with $L=\sigma=0$.
Denote the closed linear subspaces $\widetilde{\mathcal{V}}_{(0)}^{k}:=\widetilde{\mathcal{V}}^{k}\cap \mathcal{L}_{(0)}^{2}$. Thanks to the generalized Poincar\'{e}'s inequality in Lemma \ref{equivalentb}, the space $\widetilde{\mathcal{V}}^1_{(0)}$ can be equipped  with the inner product
$$\big((y,y_\Gamma),(z,z_\Gamma)\big)_{\widetilde{\mathcal{V}}^1_{(0)}}:=\int_{\Omega}\nabla y\cdot\nabla z\,\mathrm{d}x,\quad\forall\, (y,y_\Gamma),(z,z_\Gamma) \in \widetilde{\mathcal{V}}^1_{(0)},$$
and the corresponding norm $\|\cdot\|_{\widetilde{\mathcal{V}}^1_{(0)}}:=(\cdot,\cdot)_{\widetilde{\mathcal{V}}^1_{(0)}}^{1/2}$. We extend the bilinear form $a_{L,\sigma}$ to the case with $L=\sigma=0$, that is,
\begin{align*}
    a_{0,0}\big((y,y_\Gamma),(z,z_\Gamma)\big)=\int_{\Omega}\nabla y\cdot\nabla z\,\mathrm{d}x,\quad\forall\, (y,y_\Gamma),(z,z_\Gamma)\in \widetilde{\mathcal{V}}^1.
\end{align*}
Consider the operator $\mathbf{A}:\widetilde{\mathcal{V}}^1\rightarrow (\widetilde{\mathcal{V}}^{1})'$ induced by
\begin{align*}
\big\langle\mathbf{A}(y,y_\Gamma),(z,z_\Gamma) \big\rangle_{(\widetilde{\mathcal{V}}^{1})',\widetilde{\mathcal{V}}^{1}} =a_{0,0}\big((y,y_\Gamma),(z,z_\Gamma)\big),\quad\forall\,(y,y_\Gamma),(z,z_\Gamma) \in \widetilde{\mathcal{V}}^{1}.
\end{align*}
Define
$$
\widetilde{\mathcal{V}}_{(0)}^{-1}:=\big\{(y,y_{\Gamma})\in(\widetilde{\mathcal{V}}^{1})'\;:\;\overline{m}(y,y_{\Gamma})=0\big\},
$$
where the generalized mean is defined as
$$
\overline{m}(y,y_{\Gamma})=\frac{\langle (y,y_{\Gamma}),(1,1)\rangle_{(\widetilde{\mathcal{V}}^{1})', \widetilde{\mathcal{V}}^{1}}}{|\Omega|+|\Gamma|}.
$$
According to \cite{Gal}, we find that the restriction $\mathbf{A}_{0}=\mathbf{A}|_{\widetilde{\mathcal{V}}_{(0)}^{1}}$ is a linear isomorphism and its inverse $\mathbf{A}_{0}^{-1}:\widetilde{\mathcal{V}}_{(0)}^{-1}\rightarrow \widetilde{\mathcal{V}}_{(0)}^{1}$ is compact on $\mathcal{L}_{(0)}^{2}$. Besides, we can define the inner product on $\widetilde{\mathcal{V}}_{(0)}^{-1}$ by
\begin{align*}
\big((y,y_{\Gamma}),(z,z_{\Gamma})\big)_{\widetilde{\mathcal{V}}_{(0)}^{-1}} :=
a_{0,0}\big(\mathbf{A}_{0}^{-1}(y,y_{\Gamma}),\mathbf{A}_{0}^{-1}(z,z_{\Gamma})\big),\quad\forall\, (y,y_{\Gamma}),(z,z_{\Gamma})\in \widetilde{\mathcal{V}}_{(0)}^{-1}.
\end{align*}
For further properties of $\mathbf{A}$ and $\mathbf{A}_0$, we refer to \cite{Gal06,Gal} and the references therein.

Let us summarize the notations for operators and spaces corresponding to different values of the parameters $L, \sigma$ in \eqref{choice}, which will be used in the remaining part of this paper:
\begin{description}
\item[$\mathbf{(B1)}$] For $(L,\sigma)\in(0,\infty)\times(0,\infty)$,
\begin{align*}
\mathcal{A}_{L,\sigma}:=\mathfrak{S}^{L,\sigma},\quad\mathcal{H}_{L,\sigma}^{1}:=\mathcal{H}^{1},\quad \mathcal{H}_{L,\sigma,0}^{1}:=\mathcal{H}_{(0)}^{1}, \quad\mathcal{H}_{L,\sigma,0}^{-1}:=\mathcal{H}_{(0)}^{-1},
\quad\mathcal{V}_{\sigma}^{k}=\mathcal{H}^{k},
\quad\mathcal{V}_{\sigma,0}^{k}=\mathcal{H}^{k}_{(0)}.
\end{align*}
\item[$\mathbf{(B2)}$] For $(L,\sigma)\in\{0\}\times(0,\infty)$,
\begin{align*}
\mathcal{A}_{L,\sigma}:=\mathfrak{S}^{L,\sigma}, \quad\mathcal{H}_{L,\sigma}^{1}:=\mathcal{V}^{1},\quad \mathcal{H}_{L,\sigma,0}^{1}:=\mathcal{V}_{(0)}^{1}, \quad\mathcal{H}_{L,\sigma,0}^{-1}:=\mathcal{V}_{(0)}^{-1}, \quad\mathcal{V}_{\sigma}^{k}=\mathcal{V}^{k}, \quad\mathcal{V}_{\sigma,0}^{k}=\mathcal{V}^{k}_{(0)}.
\end{align*}
\item[$\mathbf{(B3)}$] For $L=\sigma=0$,
\begin{align*}
\mathcal{A}_{L,\sigma}:=\mathbf{A}_{0}^{-1}, \quad\mathcal{H}_{L,\sigma}^{1}:=\widetilde{\mathcal{V}}^{1}, \quad\mathcal{H}_{L,\sigma,0}^{1}:=\widetilde{\mathcal{V}}_{(0)}^{1}, \quad\mathcal{H}_{L,\sigma,0}^{-1}:=\widetilde{\mathcal{V}}_{(0)}^{-1}, \quad\mathcal{V}_{\sigma}^{k}=\widetilde{\mathcal{V}}^{k}, \quad\mathcal{V}_{\sigma,0}^{k}=\widetilde{\mathcal{V}}^{k}_{(0)}.
\end{align*}
\end{description}
For the sake of convenience, hereafter we shall use the bold notation for generic element $\boldsymbol{y}=(y,y_\Gamma)$ in the product spaces $\mathcal{L}^{2}$, $\mathcal{H}_{L,\sigma}^{1}$, $(\mathcal{H}_{L,\sigma}^{1})'$ etc.
Besides, the symbol $C$ stands for generic positive constants that may depend on coefficients of the system, norms of the initial datum, $\Omega$ and $\Gamma$. Their values may change from line to line and specific dependencies will be pointed out when necessary.

\subsection{The initial boundary value problem}
Set $Q:=\Omega\times(0,\infty)$, $\Sigma:=\Gamma\times(0,\infty)$ as well as $Q_{T}:=\Omega\times(0,T)$,  $\Sigma_{T}:=\Gamma\times(0,T)$ for any $T>0$.
Under the choice of parameters presented in \eqref{choice}, we reformulate the original problem \eqref{CH}--\eqref{Chini} into the following form
\begin{align}
\left\{
\begin{array}{ll}
\partial_{t}\varphi=\Delta\mu,&\quad \text{in }Q,\\
\mu=-\Delta \varphi+\beta(\varphi)+\pi(\varphi),&\quad \text{in }Q,\\
\partial_{t}\psi=\sigma\Delta_{\Gamma}\theta-\partial_{\mathbf{n}}\mu,&\quad \text{on }\Sigma,\\
\theta=\partial_{\mathbf{n}}\varphi-\Delta_{\Gamma}\psi+\beta_{\Gamma}(\psi)+\pi_{\Gamma}(\psi),&\quad \text{on }\Sigma,\\
\varphi|_{\Gamma}=\psi,&\quad \text{on }\Sigma,\\
L\partial_{\mathbf{n}}\mu=\theta-\mu|_{\Gamma},&\quad \text{on }\Sigma,\\
\varphi|_{t=0}=\varphi_{0},&\quad \text{in }\Omega,\\
\psi|_{t=0}=\psi_{0}=\varphi_{0}|_{\Gamma},&\quad \text{on }\Gamma.
\end{array}\right.
\label{model}
\end{align}
Here, we view \eqref{model} as a sort of transmission problem that consists of a Cahn-Hilliard equation for $\varphi$ in the bulk and another evolution equation for $\psi=\varphi|_\Gamma$ on the boundary \cite{KLLM,MZ}.
Besides, we make the following decomposition for the bulk and boundary potentials
$$
F=\widehat{\beta}+\widehat{\pi},\qquad G=\widehat{\beta}_\Gamma+\widehat{\pi}_\Gamma.
$$
Let us give the assumptions on the nonlinearities and initial data.
\begin{description}
\item[$\mathbf{(A1)}$] $\beta$, $\beta_{\Gamma}\in C^{1}(-1,1)$ are monotone increasing functions with
\begin{align*}
&\lim_{r\rightarrow-1}\beta(r)=-\infty,\qquad\lim_{r\rightarrow-1}\beta_{\Gamma}(r)=-\infty,\\
&\lim_{r\rightarrow1}\beta(r)=+\infty,\qquad\lim_{r\rightarrow1}\beta_{\Gamma}(r)=+\infty.
\end{align*}
Their primitive denoted by $\widehat{\beta}$, $\widehat{\beta}_{\Gamma}$, respectively, satisfy $\widehat{\beta}$, $\widehat{\beta}_{\Gamma}\in C([-1,1])\cap C^{2}(-1,1)$. The derivatives $\beta'$, $\beta_{\Gamma}'$ satisfy
\begin{align*}
    \beta'(r)\geq\varpi,\quad\beta_{\Gamma}'(r)\geq\varpi,\quad\forall \,r\in (-1,1)
\end{align*}
for some positive constant $\varpi>0$. Without loss of generality, we set $\widehat{\beta}(0)=\widehat{\beta}_{\Gamma}(0)=\beta(0)=\beta_{\Gamma}(0)=0$ and make the extension $\widehat{\beta}(r)=+\infty$, $\widehat{\beta}_{\Gamma}(r)=+\infty$ for $|r|>1$.
\item[$\mathbf{(A2)}$]  There exist positive constants $\varrho$, $c_{0}>0$ such that
  \begin{align}
|\beta(r)|\leq \varrho|\beta_{\Gamma}(r)|+c_{0},\quad\forall\,r\in (-1,1).\label{assum1}
  \end{align}
%
\item[$\mathbf{(A3)}$] $\widehat{\pi}, \widehat{\pi}_\Gamma \in C^1(\mathbb{R})$ and their derivatives $\pi:=\widehat{\pi}'$, $\pi_\Gamma:=\widehat{\pi}_\Gamma'$ are globally Lipschitz continuous with Lipschitz constants denoted by $K$ and $K_{\Gamma}$, respectively.
\item[$\mathbf{(A4)}$] $(\varphi_0,\psi_0)\in \mathcal{V}^1$ satisfying $ \widehat{\beta}(\varphi_0)\in L^1(\Omega)$, $\widehat{\beta}_\Gamma(\psi_0)\in L^1(\Gamma)$ and $\overline{m}(\varphi_0, \psi_0)=\overline{m}_{0} \in (-1,1)$.
\end{description}
\begin{remark}\rm
Following the usual setting in \cite{CC,CF15,CFW,CGS14,CGS18,LW}, we assume that $\beta_{\Gamma}$ dominates $\beta$ as in \eqref{assum1}. This compatibility condition is useful when we deal with the bulk-surface interaction and derive necessary \emph{a priori} estimates.
An alternative choice consists in taking the bulk potential as the dominating one (see \cite{GMS09}). Besides, comparing with the assumptions in \cite{FW}, here we no longer require that
the derivatives $\beta'$, $\beta_{\Gamma}'$ are convex.
\end{remark}

For the sake of convenience, we denote
$$
\bm{\varphi}=(\varphi, \psi), \quad \bm{\mu}=(\mu, \theta),\quad \bm{\beta}=(\beta,\beta_\Gamma),
\quad \bm{\pi}=(\pi, \pi_\Gamma), \quad \bm{\varphi}_0=(\varphi_0, \psi_0).
$$
Then we have the following result on the existence and uniqueness of global weak solutions to problem \eqref{model}.
\begin{proposition}
\label{existence}
Let $\Omega\subset\mathbb{R}^{d}$  $(d\in\{2,3\})$ be a bounded domain with smooth boundary $\Gamma$. Suppose that the assumptions $\mathbf{(A1)}$--$\mathbf{(A4)}$ are satisfied and the parameters $(L,\sigma)$ fulfill \eqref{choice}.  For arbitrary $T\in(0,\infty)$, problem \eqref{model} admits a global weak solution $(\boldsymbol{\varphi},\boldsymbol{\mu})$ in the following sense:
\begin{align}
&\boldsymbol{\omega}:=\boldsymbol{\varphi}-\overline{m}_{0}\boldsymbol{1}\in H^{1}(0,T;\mathcal{H}_{L,\sigma,0}^{-1})\cap L^{\infty}(0,T;\mathcal{V}_{(0)}^{1})\cap L^{2}(0,T;\mathcal{V}_{(0)}^{2}),\notag\\
&\boldsymbol{\mu}\in L^{2}(0,T;\mathcal{H}_{L,\sigma}^{1}),\qquad
\boldsymbol{\beta}(\boldsymbol{\varphi})\in L^{2}(0,T;\mathcal{L}^{2}),\notag
\end{align}
such that
\begin{align}
&\big\langle\partial_{t}\boldsymbol{\varphi}(t),\boldsymbol{y} \big\rangle_{\mathcal{H}_{L,\sigma}^{-1},\mathcal{H}_{L,\sigma}^{1}} +a_{L,\sigma}\big(\boldsymbol{\mu}(t),\boldsymbol{y}\big)=0, \qquad \forall\,\boldsymbol{y}\in\mathcal{H}_{L,\sigma}^{1},
\label{defn2.1}
\end{align}
for almost all $t\in(0,T)$, and
\begin{align}
&\mu=-\Delta \varphi+\beta(\varphi)+\pi(\varphi),
&&\text{a.e. in } Q_{T}, 
\label{eq3.5}\\
&\theta=\partial_{\mathbf{n}}\varphi-\Delta_{\Gamma}\psi+\beta_{\Gamma}(\psi)+\pi_{\Gamma}(\psi),
&&\text{a.e. on } \Sigma_{T}. 
\label{eq3.6}
\end{align}
Furthermore, the initial conditions are satisfied
\begin{align*}
\varphi|_{t=0}=\varphi_{0}\ \text{ a.e. in }\Omega,\qquad\psi|_{t=0}=\psi_{0}\ \text{ a.e. on }\Gamma.
\end{align*}
Finally, let $(\boldsymbol{\varphi}_{i}, \bm{\mu}_i)$ be two weak solutions corresponding to the initial data $\boldsymbol{\varphi}_{0,i}$, $i=1,2$, respectively. Then, there exists a constant $C>0$, depending on $K$, $K_{\Gamma}$, $\Omega$, $\Gamma$, coefficients of the system and $T$, such that
\begin{align}
&\Vert\boldsymbol{\varphi}_{1}(t)-\boldsymbol{\varphi}_{2}(t)\Vert_{\mathcal{H}_{L,\sigma,0}^{-1}}^{2}
+ \int_{0}^{t}\Vert\boldsymbol{\varphi}_{1}(s)-\boldsymbol{\varphi}_{2}(s) \Vert_{\mathcal{V}_{(0)}^{1}}^{2}\,\mathrm{d}s
\leq C\Vert\boldsymbol{\varphi}_{0,1}-\boldsymbol{\varphi}_{0,2}\Vert_{\mathcal{H}_{L,\sigma,0}^{-1}}^{2},
 \quad\forall\,t\in[0,T].
 \notag
\end{align}
As a consequence, the weak solution is unique.
\end{proposition}
\begin{remark}\label{conti}\rm
Since $T>0$ is arbitrary, the regularity of $\boldsymbol{\varphi}$ implies $\boldsymbol{\varphi}\in C_{w}([0,\infty);\mathcal{V}^{1})\cap C([0,\infty);\mathcal{L}^{2})$ so that the initial datum can be attained. A sketch of the proof for Proposition \ref{existence} can be found in Section \ref{approx}.
\end{remark}

\subsection{Statement of main results}
In this subsection, we state the main results of this paper.
The first result implies that every weak solution instantaneous regularizes for $t>0$.
\begin{theorem}
\label{eventual}
Suppose that $\Omega\subset\mathbb{R}^{d}$  $(d\in\{2,3\})$ is a bounded domain with smooth boundary $\Gamma$, the assumptions $\mathbf{(A1)}$--$\mathbf{(A4)}$ are satisfied and the parameters $(L,\sigma)$ fulfill \eqref{choice}.
Let $(\boldsymbol{\varphi},\boldsymbol{\mu})$ be the unique global weak solution to problem \eqref{model} obtained in Proposition \ref{existence}.  For any $\eta>0$, it holds
 \begin{align*}
 & \boldsymbol{\varphi}\in L^{\infty}(\eta,\infty;\mathcal{V}^{2}),\quad   \partial_t\bm{\varphi} \in L^{\infty}(\eta,\infty;\mathcal{H}_{L,\sigma,0}^{-1})\cap L^2_{\mathrm{uloc}}(\eta,\infty;\mathcal{V}_{(0)}^{1}),
 \\
 &\boldsymbol{\mu}\in L^{\infty}(\eta,\infty;\mathcal{H}_{L,\sigma}^{1})\cap L^{2}_{\mathrm{uloc}}(\eta,\infty;\mathcal{V}_{\sigma}^{2}).
 \end{align*}
 Moreover, we have $\beta_\Gamma(\psi)\in L^\infty(\eta, \infty; H_\Gamma)$ and
 $\boldsymbol{\beta}(\boldsymbol{\varphi})\in L^\infty(\eta, \infty;\mathcal{L}^q)$,
 when $d=2$, $q\in [2,\infty)$; when $d=3$, $q\in[1,6]$ if $\sigma>0$ and $q\in[1,4]$ if $\sigma=0$.
\end{theorem}

Next, we show the eventual separation property for weak solutions that holds in both two and three dimensions.

\begin{theorem}
\label{eventual-2}
Suppose that the assumptions in Theorem \ref{eventual} are satisfied. Let  $(\boldsymbol{\varphi},\boldsymbol{\mu})$ be the unique global weak solution to problem \eqref{model} obtained in Proposition \ref{existence}. There exists a constant $\delta_{1}\in(0,1)$ and a sufficiently large time $T_{\mathrm{SP}}>0$ such that
\begin{align}
\Vert \varphi(t)\Vert_{L^{\infty}(\Omega)}\leq 1-\delta_{1},\quad\Vert \psi(t)\Vert_{L^{\infty}(\Gamma)}\leq 1-\delta_{1},\quad\forall \,t\geq T_{\mathrm{SP}}.
\label{eq4.13}
\end{align}
\end{theorem}

Assume in addition, the following growth condition for the bulk potential $\beta$ near the singular points $\pm 1$ (see \cite{GP}):
\begin{description}
\item[$\mathbf{(A5)}$] As $\delta\rightarrow0^{+}$, for some $\kappa>1/2$, it holds
\begin{align}
\frac{1}{\beta(1-2\delta)}=O\Big(\frac{1}{|\ln \delta|^{\kappa}}\Big), \quad\frac{1}{|\beta(-1+2\delta)|}=O\Big(\frac{1}{|\ln \delta|^{\kappa}}\Big).
\label{addi}
\end{align}
\end{description}

With $\mathbf{(A5)}$, we can prove the instantaneous separation property for weak solutions in two dimensions.

\begin{theorem}
\label{intantaneous}
Suppose that $\Omega\subset\mathbb{R}^{2}$ is a bounded domain with smooth boundary $\Gamma$, the assumptions $\mathbf{(A1)}$--$\mathbf{(A5)}$ are satisfied and the parameters $(L,\sigma)$ fulfill \eqref{choice}. Let $(\boldsymbol{\varphi},\boldsymbol{\mu})$ be the unique global weak solution to problem \eqref{model} obtained in Proposition \ref{existence}. Then for any $\eta>0$, there exists a constant $\delta_{2}\in(0,1)$, depending on $\eta$, $\overline{m}_{0}$, $E(\boldsymbol{\varphi}_{0})$, $\Omega$ and $\Gamma$, such that
\begin{align}
\Vert \varphi(t)\Vert_{L^{\infty}(\Omega)}\leq1-\delta_{2},\quad\Vert \psi(t)\Vert_{L^{\infty}(\Gamma)}\leq1-\delta_{2},\quad\forall\;t\geq\eta.\label{4.31}
\end{align}
\end{theorem}

\begin{remark}\rm
Combining \eqref{assum1} with \eqref{addi}, we also find a growth condition on $\beta_\Gamma$.
It is straightforward to check that the logarithmic potential \eqref{logari} satisfies \eqref{addi}. Besides, with $\kappa>1/2$, the assumption $\mathbf{(A5)}$ accommodates entropy densities with
milder singularities than \eqref{logari}.
\end{remark}

\begin{remark}\label{rem-2d}\rm
The instantaneous strict separation property for the specific case
$(\mathbf{B2})$ has been obtained in \cite{FW} with the aid of the Moser-Trudinger inequality in two dimensions.
To achieve this goal, it was required in \cite{FW} that both derivatives $\beta'$, $\beta_{\Gamma}'$ are convex on $(-1,1)$, and an additional pointwise assumption was imposed on the derivative of $\beta$ apart from its strict positivity, i.e., there exists a positive constant $\widetilde{c}_{0}$ such that
\begin{align*}
|\beta'(r)|\leq e^{\widetilde{c}_{0}|\beta(r)|+\widetilde{c}_{0}},\quad\forall\,r\in(-1,1).
\end{align*}
Besides, the argument in \cite{FW} essentially relied on the assumption $\sigma>0$ since the diffusive term $\sigma\Delta_{\Gamma}\theta$ provides a regularizing effect on the boundary leading to $\theta\in V_{\Gamma}$. In Theorem \ref{intantaneous}, we establish the instantaneous separation property for all the three cases $\mathbf{(B1)}$--$\mathbf{(B3)}$ by extending the De Giorgi's iteration scheme in \cite{GP}, under a milder assumption on $\beta$ (i.e., \eqref{addi}) and without the regularizing effect $\sigma>0$ when $L=0$.
\end{remark}

The eventual separation property obtained in Theorem \ref{eventual-2} enables us to prove that every global weak solution converges to a single equilibrium as $t\to \infty$. Moreover precisely, we have
\begin{theorem}
\label{equilibrium}
Suppose that the assumptions in Theorem \ref{eventual} are satisfied. In addition, we assume that $\beta$, $\beta_{\Gamma}$ are real analytic on $(-1,1)$ and $\pi$, $\pi_{\Gamma}$ are real analytic on $\mathbb{R}$.
Let $(\boldsymbol{\varphi},\boldsymbol{\mu})$ be the unique global weak solution to problem \eqref{model} obtained in Proposition \ref{existence}.
Then it holds
\begin{align*}
\lim_{t\rightarrow\infty}\Vert\boldsymbol{\varphi}(t)-\boldsymbol{\varphi}_{\infty}\Vert_{\mathcal{V}^{2}}=0,
\end{align*}
where $\boldsymbol{\varphi}_{\infty}:=(\varphi_{\infty},\psi_{\infty})$ is a steady state satisfying the elliptic problem
\begin{align}
\left\{
\begin{array}{ll}
\mu_{\infty}=-\Delta \varphi_{\infty}+\beta(\varphi_{\infty})+\pi(\varphi_{\infty}),
&\quad \text{a.e. in }\Omega,\\
\theta_{\infty}=-\Delta_{\Gamma}\psi_{\infty}+\beta_{\Gamma}(\psi_{\infty}) +\pi_{\Gamma}(\psi_{\infty})+\partial_{\mathbf{n}}\varphi_{\infty},
&\quad \text{a.e. on }\Gamma,\\
\varphi_{\infty}|_{\Gamma}=\psi_{\infty},
&\quad \text{a.e. on }\Gamma,\\
\overline{m}(\boldsymbol{\varphi}_{\infty})=\overline{m}_{0},
\end{array}\right.
\notag
\end{align}
with constants
\begin{align}
\mu_{\infty}=\theta_{\infty}=\frac{1}{|\Omega|+|\Gamma|} \Big[\int_{\Omega}\big(\beta(\varphi_{\infty})+\pi(\varphi_{\infty})\big)\,\mathrm{d}x +\int_{\Gamma}\big(\beta_{\Gamma}(\psi_{\infty})+\pi_{\Gamma}(\psi_{\infty})\big)\,\mathrm{d}S\Big].\notag
\end{align}
 Moreover, we have the following estimate on the convergence rate
\begin{align}
\Vert\boldsymbol{\varphi}(t)-\boldsymbol{\varphi}_{\infty}\Vert_{\mathcal{V}^{1}}\leq C\big(1+t\big)^{-\frac{\theta^{*}}{1-2\theta^{*}}},\quad\forall\;t\geq 1,\label{converrate}
\end{align}
where $\theta^{*}\in(0,1/2)$ is a constant depending on $\boldsymbol{\varphi}_{\infty}$, the positive constant $C$ may depend on $E(\boldsymbol{\varphi}_{0})$, $\boldsymbol{\varphi}_{\infty}$, $\overline{m}_{0}$, $\Omega$, $\Gamma$ and coefficients of the system.
\end{theorem}

\section{Regularity of global weak solutions}
\setcounter{equation}{0}

In this section, we prove Theorem \ref{eventual} on the instantaneous regularizing property of global weak solutions for $t>0$.

\subsection{The approximating problem}
\label{approx}
For convenience of the subsequent analysis, we recall the approximating problem as in \cite{CF15,LvWu}. For every $\varepsilon\in (0,1)$, set $\beta_{\varepsilon}$, $\beta_{\Gamma,\varepsilon}:\mathbb{R}\rightarrow\mathbb{R}$, along with the associated resolvent operators $J_{\varepsilon}$, $J_{\Gamma,\varepsilon}:\mathbb{R}\rightarrow\mathbb{R}$ by
\begin{align*}
&\beta_{\varepsilon}(r):=\frac{1}{\varepsilon}\big(r-J_{\varepsilon}(r)\big) :=\frac{1}{\varepsilon}\big(r-(I+\varepsilon\beta)^{-1}(r)\big), \\
&\beta_{\Gamma,\varepsilon}(r):=\frac{1}{\varepsilon\varrho}\big(r-J_{\Gamma,\varepsilon}(r)\big) :=\frac{1}{\varepsilon\varrho}\big(r-(I+\varepsilon\varrho\beta_{\Gamma})^{-1}(r)\big),
\end{align*}
for all $r\in\mathbb{R}$, where $\varrho>0$ is the coefficient given in the condition \eqref{assum1} (see, e.g., \cite{CC,CF15}).
The related Moreau-Yosida regularizations $\widehat{\beta}_{\varepsilon}$, $\widehat{\beta}_{\Gamma,\varepsilon}$ of $\widehat{\beta}$, $\widehat{\beta}_{\Gamma}:\mathbb{R}\rightarrow\mathbb{R}$ are then given by (see, e.g., \cite{R.E.S})
\begin{align}
&\widehat{\beta}_{\varepsilon}(r) :=\inf_{s\in\mathbb{R}}\left\{\frac{1}{2\varepsilon}|r-s|^{2} +\widehat{\beta}(s)\right\}=\frac{1}{2\varepsilon}|r-J_{\varepsilon}(r)|^{2} +\widehat{\beta}\big(J_{\varepsilon}(r)\big) =\int_{0}^{r}\beta_{\varepsilon}(s)\,\mathrm{d}s,
\notag\\
&\widehat{\beta}_{\Gamma,\varepsilon}(r) :=\inf_{s\in\mathbb{R}}\left\{\frac{1}{2\varepsilon\varrho}|r-s|^{2} +\widehat{\beta}_{\Gamma}(s)\right\}=\int_{0}^{r}\beta_{\Gamma,\varepsilon}(s)\,\mathrm{d}s.
\notag
\end{align}
Define the projection operator
\begin{align*}
    \mathbf{P}:\mathcal{L}^{2}\rightarrow\mathcal{L}_{(0)}^{2},\quad \bm{y}\mapsto \bm{y}-\overline{m}(\bm{y})\bm{1}.
\end{align*}
Let us consider the following approximating problem with $\varepsilon\in (0,1)$:
\begin{align}
&\varepsilon\boldsymbol{\omega}_{\varepsilon}'(t)+\mathcal{A}_{L,\sigma}\big(\boldsymbol{\omega}_{\varepsilon}'(t)\big)+\partial\Phi\big(\boldsymbol{\omega}_{\varepsilon}(t)\big)
 =\mathbf{P}\left(-\boldsymbol{\beta}_{\varepsilon}\big(\boldsymbol{\varphi}_{\varepsilon}(t)\big)-\boldsymbol{\pi}\big(\boldsymbol{\varphi}_{\varepsilon}(t)\big)\right) &&\text{in }\mathcal{L}^{2}_{(0)},\ \text{for a.a. }t\in(0,T),\label{appro}\\
&\boldsymbol{\mu}_{\varepsilon}(t):=\varepsilon\boldsymbol{\omega}_{\varepsilon}'(t)+\partial\Phi\big(\boldsymbol{\omega}_{\varepsilon}(t)\big) +\boldsymbol{\beta}_{\varepsilon}\big(\boldsymbol{\varphi}_{\varepsilon}(t)\big)+\boldsymbol{\pi}\big(\boldsymbol{\varphi}_{\varepsilon}(t)\big) &&\text{in }\mathcal{L}^{2},\ \text{for a.a. }t\in(0,T),\label{2.12}\\
&\boldsymbol{\omega}_{\varepsilon}(0)=\boldsymbol{\omega}_{0} =\boldsymbol{\varphi}_{0}-\overline{m}_{0}\boldsymbol{1}&& \text{in }\mathcal{L}_{(0)}^{2},\label{approini}
\end{align}
where $\boldsymbol{\varphi}_{\varepsilon}=\boldsymbol{\omega}_{\varepsilon}+\overline{m}_{0}\boldsymbol{1}$
and $\boldsymbol{\beta}_{\varepsilon}=(\beta_{\varepsilon},\beta_{\Gamma,\varepsilon})$.
In \eqref{appro} and \eqref{2.12}, $\Phi:\mathcal{L}_{(0)}^{2}\rightarrow[0,\infty]$ is a proper, lower semi-continuous and convex functional such that
\begin{align}
\Phi(\boldsymbol{y}):=\left\{
    \begin{array}{ll}
     \displaystyle{\frac{1}{2}\int_{\Omega}|\nabla y|^{2}\,\mathrm{d}x+ \frac{1}{2}\int_{\Gamma}|\nabla_{\Gamma} y_{\Gamma}|^{2}\,\mathrm{d}S,}   &\text{if }\boldsymbol{y}\in\mathcal{V}_{(0)}^{1},  \\[2mm]
     +\infty,    & \text{otherwise}.
    \end{array}\right.\notag
\end{align}
The subdifferential $\partial\Phi$ of $\Phi$ is given by $\partial\Phi(\boldsymbol{y})=(-\Delta y,\partial_{\mathbf{n}}y-\Delta_{\Gamma}y_{\Gamma})$ with $D(\partial\Phi)=\mathcal{V}_{(0)}^{2}$ (see, e.g., \cite[Lemma C]{CF15}).
\smallskip

\begin{proof}[\textbf{A sketch of the proof for Proposition \ref{existence}}] Applying the abstract theory of doubly nonlinear evolutions (cf. \cite{CV}) and the contraction mapping principle, we can conclude that there exists a unique solution
$$
\boldsymbol{\omega}_{\varepsilon}\in H^{1}(0,T;\mathcal{L}_{(0)}^{2})\cap L^{\infty}(0,T;\mathcal{V}_{(0)}^{1})\cap L^{2}(0,T;\mathcal{V}_{(0)}^{2})
$$
satisfying \eqref{appro}--\eqref{approini} (cf. \cite{CF15,LvWu}). By \eqref{2.12}, we further obtain  $\boldsymbol{\mu}_{\varepsilon}\in L^{2}(0,T;\mathcal{H}_{L,\sigma}^{1})$.
In this manner, one can obtain a family of approximating solutions $(\boldsymbol{\varphi}_{\varepsilon},\boldsymbol{\mu}_{\varepsilon})$ satisfying sufficient \emph{a priori} estimates that are uniform with respect to the approximating parameter $\varepsilon$. By the compactness argument and passing to the limit as $\varepsilon\rightarrow 0$ (up to a subsequence), we can find a limit pair $(\boldsymbol{\varphi},\boldsymbol{\mu})$ that gives the global weak solution to the original problem \eqref{model} on $[0,T]$. Uniqueness of $(\boldsymbol{\varphi},\boldsymbol{\mu})$ follows from the standard energy method.
For the case with $(L,\sigma)\in(0,\infty)\times(0,\infty)$, the conclusion of Proposition \ref{existence} follows from the same proof as in \cite{LvWu}, while for $(L,\sigma)\in\{0\}\times[0,\infty)$, we refer to \cite{CF15}. We just mention that in the current problem \eqref{model}, those external sources considered in \cite{CF15,LvWu} are neglected so that we are able to derive uniform-in-time estimates and extend the solution to the whole interval $[0,\infty)$ (cf. \cite{FW} for the special case with $(L,\sigma)\in \{0\}\times [0,\infty)$).
\end{proof}

\subsection{Mass conservation and energy equality}

We present some basic properties of global weak solutions to problem \eqref{model}.

\begin{proposition}[Mass conservation]
Let $\boldsymbol{\varphi}$ be the global weak solution obtained in Proposition \ref{existence}.
It holds
\begin{align}
\overline{m}(\boldsymbol{\varphi}(t))=\overline{m}_{0},\quad \forall\, t\geq 0.
\label{Mass-conser}
\end{align}
\end{proposition}
\begin{proof}
Taking $\boldsymbol{y}=\boldsymbol{1}$ in the weak formulation \eqref{defn2.1}, we easily arrive at the conclusion \eqref{Mass-conser}, see \cite[Remark 2]{CF15} for $L=0$  and \cite[Lemma 3.1]{LvWu} for $L>0$.
\end{proof}

Next, we derive an energy inequality for global weak solutions.

\begin{lemma}
Let $(\boldsymbol{\varphi}, \boldsymbol{\mu})$ be the global weak solution obtained in Proposition \ref{existence}.
It holds
\begin{align}
E\big(\boldsymbol{\varphi}(t)\big)+\int_{0}^{t}\Vert\boldsymbol{\varphi}'(s)\Vert_{\mathcal{H}_{L,\sigma,0}^{-1}}^{2}\,\mathrm{d}s\leq E\big(\boldsymbol{\varphi}_{0}\big),\quad \text{for a.a.}\ t\geq 0.\label{energyinequ}
\end{align}
Moreover, there exists a positive constant $M_{1}$ such that
\begin{align}
&\Vert\boldsymbol{\varphi}\Vert_{L^{\infty}(0,\infty;\mathcal{V}^{1})}+\int_{0}^{\infty}\Vert\boldsymbol{\varphi}'(t)\Vert_{\mathcal{H}_{L,\sigma,0}^{-1}}^{2}\,\mathrm{d}t\leq M_{1},\label{uni1}\\
&\int_{0}^{\infty}\Vert\mathbf{P}\boldsymbol{\mu}(t)\Vert_{\mathcal{H}_{L,\sigma,0}^{1}}^{2}\mathrm{d}t\leq M_{1}.\label{uni2}
\end{align}
\end{lemma}
\begin{proof}
We apply the argument in \cite{FW} for the case $(L,\sigma)\in\{0\}\times [0,\infty)$. Testing \eqref{appro} by $\boldsymbol{\omega}'_{\varepsilon}\in L^{2}(0,\infty;\mathcal{L}_{(0)}^{2})$, using the chain rule of subdifferential (see, e.g., \cite{R.E.S}), we find that $t\mapsto E_{\varepsilon}\big(\boldsymbol{\varphi}_{\varepsilon}(t)\big)$ is absolutely continuous on $[0,\infty)$ and
\begin{align}
\frac{\mathrm{d}}{\mathrm{d}t}E_{\varepsilon}\big(\boldsymbol{\varphi}_{\varepsilon}(t)\big)
+ \varepsilon\Vert\boldsymbol{\omega}'_{\varepsilon}(t)\Vert_{\mathcal{L}^{2}}^{2} +\Vert\boldsymbol{\omega}'_{\varepsilon}(t)\Vert_{\mathcal{H}_{L,\sigma,0}^{-1}}^{2} =0, \quad\text{for a.a.} \ t>0,
\label{L1}
\end{align}
where
\begin{align}
E_{\varepsilon}\big(\boldsymbol{y}\big)&=\frac{1}{2}\int_{\Omega}|\nabla y|^{2}\,\mathrm{d}x +\frac{1}{2}\int_{\Gamma}|\nabla_{\Gamma}y_{\Gamma}|^{2}\,\mathrm{d}S +\int_{\Omega}\widehat{\beta}_{\varepsilon}(y)+\widehat{\pi}(y)\,\mathrm{d}x \notag\\
&\quad+\int_{\Gamma}\widehat{\beta}_{\Gamma,\varepsilon}(y_{\Gamma}) +\widehat{\pi}_{\Gamma}(y_{\Gamma})\,\mathrm{d}S,
\quad\forall\,\boldsymbol{y}\in \mathcal{V}^{1}.
\notag
\end{align}
Integrating \eqref{L1} with respective to time, we obtain
\begin{align}
E_{\varepsilon}\big(\boldsymbol{\varphi}_{\varepsilon}(t)\big) +
\varepsilon\int_{0}^{t}\Vert\boldsymbol{\omega}'_{\varepsilon}(s)\Vert_{\mathcal{L}^{2}}^{2}\,\mathrm{d}s +\int_{0}^{t}\Vert\boldsymbol{\omega}'_{\varepsilon}(s)\Vert_{\mathcal{H}_{L,\sigma,0}^{-1}}^{2}\,\mathrm{d}s  =E_{\varepsilon}\big(\boldsymbol{\varphi}_{0}\big),\quad \forall\, t\geq 0.\label{eq4.1}
\end{align}
Recall the weak and strong convergence results (in the sense of subsequence)
\begin{align*}
\boldsymbol{\omega}_{\varepsilon}\rightarrow\boldsymbol{\omega}&\quad\text{weakly star in }L^{\infty}(0,T;\mathcal{V}_{(0)}^{1}),\\
\boldsymbol{\omega}_{\varepsilon}\rightarrow\boldsymbol{\omega}&\quad\text{weakly in }H^{1}(0,T;\mathcal{H}_{L,\sigma,0}^{-1})\cap L^{2}(0,T;\mathcal{V}^{2}),\\
\boldsymbol{\omega}_{\varepsilon}\rightarrow\boldsymbol{\omega}&\quad\text{strongly in }C([0,T];\mathcal{L}_{(0)}^{2})\cap L^{2}(0,T;\mathcal{V}_{(0)}^{1}),\\
\varepsilon\boldsymbol{\omega}_{\varepsilon}\rightarrow\boldsymbol{0}&\quad\text{strongly in }H^{1}(0,T;\mathcal{L}^{2}),
\end{align*}
for any $T>0$ (see \cite{CF15,LvWu,FW}).
By the lower weak semicontinuity of norms, we get
\begin{align*}
\liminf_{\varepsilon\rightarrow0}\,\int_{0}^{t}\Vert\boldsymbol{\omega}'_{\varepsilon}(s)\Vert_{\mathcal{H}_{L,\sigma,0}^{-1}}^{2}\,\mathrm{d}s \geq\int_{0}^{t}\Vert\boldsymbol{\omega}'(s)\Vert_{\mathcal{H}_{L,\sigma,0}^{-1}}^{2}\,\mathrm{d}s, \quad\forall\,t>0.
\end{align*}
Besides, it is straightforward to check that
\begin{align*}
&\liminf_{\varepsilon\rightarrow0}E_{\varepsilon}\big(\boldsymbol{\varphi}_{\varepsilon}(t)\big) \geq  E\big(\boldsymbol{\varphi}(t)\big)\quad\text{for a.a. }t>0,
\quad\text{and}\quad
 \lim_{\varepsilon\rightarrow0} E_{\varepsilon}\big(\boldsymbol{\varphi}_0\big)= E\big(\boldsymbol{\varphi}_0\big).
\end{align*}
Thus, taking lim inf as $\varepsilon\rightarrow0$ in \eqref{eq4.1}, we obtain the energy
inequality \eqref{energyinequ}.

Thanks to $\mathbf{(A1)}$, $\mathbf{(A3)}$, we can find a nonnegative constant $c_{1}$ such that
\begin{align}
\widehat{\beta}(r)+\widehat{\pi}(r)\geq-c_{1}, \quad\widehat{\beta}_{\Gamma}(r)+\widehat{\pi}_{\Gamma}(r)\geq-c_{1}, \quad\forall\,r\in[-1,1].\notag
\end{align}
Hence, by the definition of $E(\boldsymbol{\varphi})$, we get
\begin{align}
E\big(\boldsymbol{\varphi}(t)\big)\geq\frac{1}{2}\int_{\Omega}|\nabla\varphi(t)|^{2}\,\mathrm{d}x+\frac{1}{2}\int_{\Gamma}|\nabla_{\Gamma}\psi(t)|^{2}\,\mathrm{d}S-c_{1}\big(|\Omega|+|\Gamma|\big),\quad\text{for a.a. }t\geq0.
\label{e-lower}
\end{align}
Combining it with \eqref{energyinequ}, recalling the generalized Poincar\'e's inequality in Lemma \ref{equivalent}, we see that
\begin{align*}
\int_{0}^{t}\Vert\boldsymbol{\varphi}'(s)\Vert_{\mathcal{H}_{L,\sigma,0}^{-1}}^{2}\,\mathrm{d}s\leq E\big(\boldsymbol{\varphi}_{0}\big)+c_{1}\big(|\Omega|+|\Gamma|\big),
\end{align*}
and
\begin{align*}
\Vert\boldsymbol{\varphi}(t)\Vert_{\mathcal{V}^{1}}^{2}&\leq2(c_{P}^{2}+1)\Vert\boldsymbol{\varphi}(t)-\overline{m}_{0}\boldsymbol{1}\Vert_{\mathcal{V}_{(0)}^{1}}^{2}+2\big(|\Omega|+|\Gamma|\big)|\overline{m}_{0}|^{2}\\
&\leq4(c_{P}^{2}+1)E(\boldsymbol{\varphi}_{0})+2(c_{P}^{2}+1)c_{1}\big(|\Omega|+|\Gamma|\big)+2\big(|\Omega|+|\Gamma|\big)|\overline{m}_{0}|^{2}
\end{align*}
for almost all $t\geq0$. The above estimates yield \eqref{uni1}, and for the case $L=0$, $\sigma\in [0,\infty)$, we refer to \cite{FW}. Finally, by the definition of the operator $\mathcal{A}_{L,\sigma}$ and \eqref{defn2.1}, we have
\begin{align}
-\mathcal{A}_{L,\sigma}\big(\boldsymbol{\varphi}'(t)\big)=\mathbf{P}\boldsymbol{\mu}(t)\quad\mathrm{in\;}\mathcal{H}_{L,\sigma,0}^{1},\quad\text{for a.a. }t\geq0, \label{equi}
\end{align}
which together with \eqref{uni1} allows us to conclude \eqref{uni2}.
\end{proof}

The following lemma implies that the time derivative of weak solution regularizes for positive time.

\begin{lemma}
\label{time-regu}
Let $\boldsymbol{\varphi}$ be the global weak solution obtained in Proposition \ref{existence}.
For any given $\eta>0$, there exists a positive constant $M_{2}$ such that
\begin{align}
\Vert\boldsymbol{\varphi}'\Vert_{L^{\infty}(\eta,\infty;\mathcal{H}_{L,\sigma,0}^{-1})}+\int_{t}^{t+1}\Vert\boldsymbol{\varphi}'(s)\Vert_{\mathcal{V}_{(0)}^{1}}^{2}\,\mathrm{d}s\leq M_{2},\quad\forall\,t\geq\eta.
\label{4.6}
\end{align}
\end{lemma}
\begin{proof}
Consider the difference between \eqref{appro} at time $s$ and $s+h$ ($h>0$), that is
\begin{align}
&\varepsilon\big(\boldsymbol{\omega}_{\varepsilon}'(s+h)-\boldsymbol{\omega}_{\varepsilon}'(s)\big) +\mathcal{A}_{L,\sigma}\big(\boldsymbol{\omega}_{\varepsilon}'(s+h)-\boldsymbol{\omega}_{\varepsilon}'(s)\big) +\partial\Phi\big(\boldsymbol{\omega}_{\varepsilon}(s+h)-\boldsymbol{\omega}_{\varepsilon}(s)\big)\notag\\
&\quad=\mathbf{P}\left(-\boldsymbol{\beta}_{\varepsilon}\big(\boldsymbol{\varphi}_{\varepsilon}(s+h)\big) +\boldsymbol{\beta}_{\varepsilon}\big(\boldsymbol{\varphi}_{\varepsilon}(s)\big)\right) +\mathbf{P}\left(-\boldsymbol{\pi}\big(\boldsymbol{\varphi}_{\varepsilon}(s+h)\big) +\boldsymbol{\pi}\big(\boldsymbol{\varphi}_{\varepsilon}(s)\big)\right) \quad\text{in }\mathcal{L}_{(0)}^{2}
\label{3.2-1}
\end{align}
for almost all $s\geq 0$. Multiplying \eqref{3.2-1} by $\boldsymbol{\omega}_{\varepsilon}(s+h)-\boldsymbol{\omega}_{\varepsilon}(s)$, dividing the resultant by $h^{2}$, according to the monotonicity of $\beta_{\varepsilon}$, $\beta_{\Gamma,\varepsilon}$, the Lipschitz continuity of $\pi$, $\pi_{\Gamma}$ and the Ehrling's lemma (see Lemma \ref{Ehrling}), we have
\begin{align}
&\frac{\varepsilon}{2}\frac{\mathrm{d}}{\mathrm{d}s}\Big\Vert\frac{\boldsymbol{\omega}_{\varepsilon}(s+h)-\boldsymbol{\omega}_{\varepsilon}(s)}{h}\Big\Vert_{\mathcal{L}^{2}}^{2}+\frac{1}{2}\frac{\mathrm{d}}{\mathrm{d}s}\Big\Vert\frac{\boldsymbol{\omega}_{\varepsilon}(s+h)-\boldsymbol{\omega}_{\varepsilon}(s)}{h}\Big\Vert_{\mathcal{H}_{L,\sigma,0}^{-1}}^{2}+\Big\Vert\frac{\boldsymbol{\omega}_{\varepsilon}(s+h)-\boldsymbol{\omega}_{\varepsilon}(s)}{h}\Big\Vert_{\mathcal{V}_{(0)}^{1}}^{2}\notag\\
&\quad\leq(K+K_{\Gamma})\Big\Vert\frac{\boldsymbol{\omega}_{\varepsilon}(s+h)-\boldsymbol{\omega}_{\varepsilon}(s)}{h}\Big\Vert_{\mathcal{L}^{2}}^{2}\notag\\
&\quad\leq\frac{1}{2}\Big\Vert\frac{\boldsymbol{\omega}_{\varepsilon}(s+h) -\boldsymbol{\omega}_{\varepsilon}(s)}{h}\Big\Vert_{\mathcal{V}_{(0)}^{1}}^{2} +C\Big\Vert\frac{\boldsymbol{\omega}_{\varepsilon}(s+h) -\boldsymbol{\omega}_{\varepsilon}(s)}{h}\Big\Vert_{\mathcal{H}_{L,\sigma,0}^{-1}}^{2},
\label{diff1}
\end{align}
where the positive constant $C$ depends on $K$, $K_{\Gamma}$, $L$ and $\sigma$.
Since the specific form of $\mathcal{A}_{L,\sigma}$ does not play an essential role here, we can apply the same argument as in \cite[Lemma 3.3]{FW} to conclude \eqref{4.6}, the details are omitted.
\end{proof}

Now we are able to improve the energy inequality \eqref{energyinequ} to be an energy equality.

\begin{proposition}[Energy equality]
\label{energyequality}
 Let $\boldsymbol{\varphi}$ be the global weak solution obtained in Proposition \ref{existence}. We have
\begin{align}
\frac{\mathrm{d}}{\mathrm{d}t}E\big(\boldsymbol{\varphi}(t)\big)+\Vert\boldsymbol{\varphi}'(t) \Vert_{\mathcal{H}_{L,\sigma,0}^{-1}}^{2}=0,\quad\text{for a.a. }t> 0, \label{4.11}
\end{align}
and
\begin{align}
E\big(\boldsymbol{\varphi}(t)\big)+\int_{0}^{t}\Vert\boldsymbol{\varphi}'(s) \Vert_{\mathcal{H}_{L,\sigma,0}^{-1}}^{2}\,\mathrm{d}s=E\big(\boldsymbol{\varphi}_{0}\big), \quad\forall\,t\geq 0.\label{Energyequ}
\end{align}
\end{proposition}
\begin{proof}
For any given $\eta>0$, we infer from \eqref{equi}, Lemma \ref{time-regu}, the regularity properties obtained in Proposition \ref{existence} and the chain rule of subdifferentials (see, e.g., \cite{R.E.S}) that
\begin{align}
-\int_{s}^{t}\Vert\boldsymbol{\varphi}'(\tau)\Vert_{\mathcal{H}_{L,\sigma,0}^{-1}}^{2}\,\mathrm{d}\tau
&=\int_{s}^{t}\big(\boldsymbol{\varphi}'(\tau),\boldsymbol{\mu}(\tau)\big)_{\mathcal{L}^{2}}\,\mathrm{d}\tau\notag\\
&=\int_{s}^{t}\big(\boldsymbol{\varphi}'(\tau),\partial\Phi\big(\boldsymbol{\varphi}(\tau)\big)\big)_{\mathcal{L}^{2}}\,\mathrm{d}\tau+\int_{s}^{t}\big(\boldsymbol{\varphi}'(\tau),\boldsymbol{\beta}\big(\boldsymbol{\varphi}(\tau)\big)+\boldsymbol{\pi}\big(\boldsymbol{\varphi}(\tau)\big)\big)_{\mathcal{L}^{2}}\,\mathrm{d}\tau\notag\\
&=\int_{s}^{t}\frac{\mathrm{d}}{\mathrm{d}\tau}E\big(\boldsymbol{\varphi}(\tau)\big)\,\mathrm{d}\tau\notag\\
&=E\big(\boldsymbol{\varphi}(t)\big)-E\big(\boldsymbol{\varphi}(s)\big),\quad\forall\,s,t\geq\eta.
\label{Energyequb}
\end{align}
Hence, the mapping $t\mapsto E\big(\boldsymbol{\varphi}(t)\big)$ is absolutely continuous and non-increasing for all $t\geq \eta$.
Since $\eta>0$ is arbitrary, the energy identity \eqref{4.11} holds for almost all $t>0$.
It follows from \eqref{energyinequ} that $\limsup_{s\to 0} E\big(\boldsymbol{\varphi}(s)\big)\leq E\big(\boldsymbol{\varphi}(0)\big)$. On the other hand, from Remark \ref{conti}, lower weak semicontinuity of norms and Lebesgue's dominated convergence theorem, we have $\liminf_{s\to 0} E\big(\boldsymbol{\varphi}(s)\big)\geq E\big(\boldsymbol{\varphi}(0)\big)$.
As a result, it holds $\lim_{s\to 0}E\big(\boldsymbol{\varphi}(s)\big)=E\big(\boldsymbol{\varphi}(0)\big)$. This together with \eqref{uni1} enables us to pass to the limit as $s\to 0$ in  \eqref{Energyequb} and conclude the energy equality \eqref{Energyequ}.
\end{proof}
\begin{remark}\rm
Proposition \ref{energyequality} implies that the mapping $t\mapsto E\big(\boldsymbol{\varphi}(t)\big)$ is absolutely continuous for all $t\geq0$. Then applying the same argument as in \cite[Proposition 4.1]{FW}, we find $\boldsymbol{\varphi}\in C([0,\infty);\mathcal{V}^{1})$.
\end{remark}
\begin{corollary}\label{decay1}
Let $\boldsymbol{\varphi}$ be the global weak solution obtained in Proposition \ref{existence}. There exists some constant $E_{\infty}\in \mathbb{R}$ such that $$\lim_{t\rightarrow\infty}E(\boldsymbol{\varphi}(t))=E_{\infty}\quad \text{and} \quad
\lim_{t\to\infty} \int_{t}^{t+1}\Vert\boldsymbol{\varphi}'(s)\Vert_{\mathcal{H}_{L,\sigma,0}^{-1}}^{2}\,\mathrm{d}s=0.
$$
\end{corollary}
\begin{proof}
Since the free energy $E(\boldsymbol{\varphi}(\cdot))$ is monotone non-increasing in time and bounded from below (see \eqref{e-lower}), the conclusion is an easy consequence of the energy equality \eqref{Energyequ}.
\end{proof}

\subsection{Proof of Theorem \ref{eventual}}

We proceed to derive some regularity properties of the weak solution.

\begin{lemma}\label{regh1}
Let $(\boldsymbol{\varphi}, \boldsymbol{\mu})$ be the global weak solution obtained in Proposition \ref{existence}.
For any given $\eta>0$, there exists a positive constant $M_{3}$ such that
\begin{align}
&\Vert\beta(\varphi)\Vert_{L^{\infty}(\eta,\infty;L^{1}(\Omega))}+\Vert\beta_{\Gamma}(\psi)\Vert_{L^{\infty}(\eta,\infty;L^{1}(\Gamma))}\leq M_{3},\label{4.12}\\
&\Vert\boldsymbol{\mu}\Vert_{L^{\infty}(\eta,\infty;\mathcal{H}_{L,\sigma}^{1})}\leq M_{3}.\label{4.13}
\end{align}
\end{lemma}
\begin{proof}
The case $(L,\sigma)\in\{0\}\times [0,\infty)$ has been treated in \cite[Lemma 3.4]{FW}. Below we extend the result to the general case \eqref{choice}. According to \eqref{equi}, \eqref{4.6}, it holds
\begin{align}
\Vert\mathbf{P}\boldsymbol{\mu}\Vert_{L^{\infty}(\eta,\infty;\mathcal{H}_{L,\sigma,0}^{1})}=\Vert\boldsymbol{\varphi}'\Vert_{L^{\infty}(\eta,\infty;\mathcal{H}_{L,\sigma,0}^{-1})}\leq M_{2}.\notag
\end{align}
Then, by the generalized Poincar\'e's inequality (see Lemma \ref{equivalent}), we get
\begin{align*}
\Vert\boldsymbol{\mu}\Vert_{\mathcal{H}_{L,\sigma}^{1}} \leq\big(|\Omega|+|\Gamma|\big)^{1/2}\overline{m}(\boldsymbol{\mu}) +\Vert\mathbf{P}\boldsymbol{\mu}\Vert_{\mathcal{H}_{L,\sigma}^{1}} \leq\big(|\Omega|+|\Gamma|\big)^{1/2}\overline{m}(\boldsymbol{\mu}) +2c_{P}\Vert\mathbf{P}\boldsymbol{\mu}\Vert_{\mathcal{H}_{L,\sigma,0}^{1}}.
\end{align*}
In order to obtain \eqref{4.13}, it remains to control the mean $\overline{m}(\boldsymbol{\mu})$. Taking $\boldsymbol{z}=\mathcal{A}_{L,\sigma}\big(\boldsymbol{\omega}(s)\big)$ in \eqref{defn2.1} yields
\begin{align}
\big(\boldsymbol{\omega}'(s),\boldsymbol{\omega}(s)\big)_{\mathcal{H}_{L,\sigma,0}^{-1}} &=\Big(\mathcal{A}_{L,\sigma}\big(\boldsymbol{\omega}'(s)\big),\mathcal{A}_{L,\sigma}\big(\boldsymbol{\omega}(s)\big)\Big)_{\mathcal{H}_{L,\sigma,0}^{1}}\notag\\
&=-\Big(\mathbf{P}\boldsymbol{\mu}(s),\mathcal{A}_{L,\sigma}\big(\boldsymbol{\omega}(s)\big)\Big)_{\mathcal{H}_{L,\sigma,0}^{1}}\notag\\
&=-\big\langle\boldsymbol{\omega}(s),\mathbf{P}\boldsymbol{\mu}(s)\big\rangle_{\mathcal{H}_{L,\sigma,0}^{-1},\mathcal{H}_{L,\sigma,0}^{1}}\notag\\
&=-\big(\boldsymbol{\omega}(s),\mathbf{P}\boldsymbol{\mu}(s)\big)_{\mathcal{L}^{2}}\notag\\
&=-\big(\boldsymbol{\omega}(s),\boldsymbol{\mu}(s)\big)_{\mathcal{L}^{2}}\notag\\
&=-\Vert\boldsymbol{\omega}(s)\Vert_{\mathcal{V}^{1}_{(0)}}^{2} -\Big(\boldsymbol{\beta}\big(\boldsymbol{\varphi}(s)\big), \boldsymbol{\omega}(s)\Big)_{\mathcal{L}^{2}} -\Big(\boldsymbol{\pi}\big(\boldsymbol{\varphi}(s)\big), \boldsymbol{\omega}(s)\Big)_{\mathcal{L}^{2}},
\label{eq4.3}
\end{align}
for almost all $s>0$. Due to the assumptions $\mathbf{(A1)}$, $\mathbf{(A4)}$, there exist positive constants $c_{3}$ and $c_{4}$ such that (cf. \cite{MZ04,GMS09})
\begin{align}
&\beta(r)(r-\overline{m}_{0})\geq c_{3}|\beta(r)|-c_{4},\quad
\beta_{\Gamma}(r)(r-\overline{m}_{0})\geq c_{3}|\beta_{\Gamma}(r)|-c_{4},\qquad\forall\,r\in(-1,1).\label{4.16}
\end{align}
Recalling the fact  $\boldsymbol{\omega}=\boldsymbol{\varphi}-\overline{m}_{0}\boldsymbol{1}$, we infer from \eqref{eq4.3}--\eqref{4.16} that
\begin{align}
&c_{3}\int_{\Omega}\big|\beta\big(\varphi(s)\big)\big|\,\mathrm{d}x+c_{3}\int_{\Gamma}\big|\beta_{\Gamma}\big(\psi(s)\big)\big|\,\mathrm{d}S\notag\\
&\quad\leq c_{4}\big(|\Omega|+|\Gamma|\big)+\Vert\boldsymbol{\pi}\big(\boldsymbol{\varphi}(s)\big)\Vert_{\mathcal{L}^{2}}\Vert\boldsymbol{\omega}(s)\Vert_{\mathcal{L}^{2}}+\Vert\boldsymbol{\omega}'(s)\Vert_{\mathcal{H}_{L,\sigma,0}^{-1}}\Vert\boldsymbol{\omega}(s)\Vert_{\mathcal{H}_{L,\sigma,0}^{-1}}\notag\\
&\quad\leq c_{4}\big(|\Omega|+|\Gamma|\big)+\Big(\big(K+K_{\Gamma}\big)\Vert\boldsymbol{\varphi}(s)\Vert_{\mathcal{L}^{2}}+\Vert\boldsymbol{\pi}(\boldsymbol{0})\Vert_{\mathcal{L}^{2}}\Big)\Big(\Vert\boldsymbol{\varphi}(s)\Vert_{\mathcal{L}^{2}}+|\overline{m}_{0}|\big(|\Omega|+|\Gamma|\big)^{1/2}\Big)\notag\\
&\qquad+C\Vert\boldsymbol{\varphi}'(s)\Vert_{\mathcal{H}_{L,\sigma,0}^{-1}}\big(\Vert\boldsymbol{\varphi}(s)\Vert_{\mathcal{L}^{2}}+1\big),\quad\text{for a.a. }s\geq\eta,\notag
\end{align}
which together with \eqref{uni1}, \eqref{4.6} yields \eqref{4.12}.
Thus, by \eqref{eq3.5}, \eqref{eq3.6}, we obtain
\begin{align}
|\overline{m}(\boldsymbol{\mu}(s))|\leq\frac{1}{|\Omega|+|\Gamma|}\Big(\int_{\Omega}\big|\beta\big(\varphi(s)\big)\big|\,\mathrm{d}x+\int_{\Gamma}\big|\beta_{\Gamma}\big(\psi(s)\big)\big|\,\mathrm{d}S+\int_{\Omega}\big|\pi\big(\varphi(s)\big)\big|\,\mathrm{d}x+\int_{\Gamma}\big|\pi_{\Gamma}\big(\psi(s)\big)\big|\,\mathrm{d}S\Big),\label{4.16'}
\end{align}
for almost all $s\geq\eta$. Combining \eqref{uni1}, \eqref{4.12} and \eqref{4.16'}, we arrive at the conclusion \eqref{4.13}.
\end{proof}
\begin{remark}\label{regh1r}\rm
Thanks to \eqref{4.6}, \eqref{4.13}, we can view \eqref{defn2.1} as an elliptic problem for $\boldsymbol{\mu}$ such that
\begin{align}
a_{L,\sigma}\big(\boldsymbol{\mu}(t),\boldsymbol{z}\big)=-\big(\boldsymbol{\omega}'(t),\boldsymbol{z}\big)_{\mathcal{L}^{2}} ,\qquad\text{for all }\boldsymbol{z}\in\mathcal{H}_{L,\sigma}^{1}\text{ and a.a. }t\in[\eta,\infty).\notag
\end{align}
From the elliptic regularity theorem, we find
$\boldsymbol{\mu}\in L^{2}_{\mathrm{uloc}}(\eta,\infty;\mathcal{V}_{\sigma}^{2})$.
This implies that $(\bm{\varphi}, \bm{\mu})$ becomes a strong solution of problem \eqref{model} on $(0,\infty)$, since $\eta>0$ is arbitrary.
\end{remark}
\begin{lemma}\label{regh2}
Let $\boldsymbol{\varphi}$ be the global weak solution obtained in Proposition \ref{existence}.

(1) Let $d=2$. For any given $\eta>0$, $p\in[2,\infty)$, there exists a positive constant $\widetilde{C}_{0}$ such that
\begin{align}
\Vert\beta(\varphi)\Vert_{L^{\infty}(\eta,\infty;L^{p}(\Omega))}+\Vert\beta(\psi)\Vert_{L^{\infty}(\eta,\infty;L^{p}(\Gamma))}\leq\widetilde{C}_{0}\sqrt{p},\label{eq4.4}
\end{align}
where $\widetilde{C}_{0}$ is independent of $p$.

(2) Let $d=3$. For any given $\eta>0$, there exists a positive constant $\widetilde{C}_1$ such that
\begin{align}
\Vert\beta(\varphi)\Vert_{L^{\infty}(\eta,\infty;L^{p}(\Omega))} +\Vert\beta(\psi)\Vert_{L^{\infty}(\eta,\infty;L^{p}(\Gamma))}\leq \widetilde{C}_1,\notag
\end{align}
where $p\in[1,6]$ if $\sigma>0$ and $p\in[1,4]$ if $\sigma=0$, the constant $\widetilde{C}_{1}$ may depend on $p$.
\end{lemma}
\begin{proof}
The proof is similar to that for \cite[(3.27)]{FW}, where the case $(L,\sigma)\in\{0\}\times [0,\infty)$ has been treated. In particular, the compatibility condition $\mathbf{(A2)}$ is essentially used.
Here, only two main differences have to be taken into account: (1) when $d=2$, one needs to apply Lemma \ref{critical} to control the $L^p$-norm; (2) when $L>0$, the boundary chemical potential $\theta$ is no longer the trace of the bulk chemical potential $\mu$. The details are omitted.
\end{proof}

\begin{lemma}\label{regh3}
Let $\boldsymbol{\varphi}$ be the global weak solution obtained in Proposition \ref{existence}.
For any given $\eta>0$, there exists a positive constant $M_{4}$ such that
\begin{align}
& \Vert\boldsymbol{\varphi}\Vert_{L^{\infty}(\eta,\infty;\mathcal{V}^{2})}\leq M_{4} ,\qquad \Vert\beta_{\Gamma}(\psi)\Vert_{L^{\infty}(\eta,\infty;H_{\Gamma})}\leq M_{4}.
\label{4.29}
\end{align}
\end{lemma}
\begin{proof}
Applying the same argument for \cite[(3.40)]{FW}, i.e., the case $(L,\sigma)\in\{0\}\times [0,\infty)$, we can deduce that
\begin{align}
\Vert\beta_{\Gamma}(\psi)\Vert_{H_{\Gamma}}^{2} \leq2\Vert \theta-\pi_\Gamma(\psi)\Vert_{H_{\Gamma}}^{2} +2\gamma\Vert\varphi\Vert_{H^{2}(\Omega)}^{2}+2C_{\gamma}\Vert\varphi\Vert_{H}^{2},\label{4.35}
\end{align}
for any $\gamma>0$, where the positive constant $C_\gamma$ depends on $\gamma$.
Next, it follows from the elliptic regularity theory (see, e.g., \cite[Corollary A.1]{MZ}) that
\begin{align}
\Vert\boldsymbol{\varphi}\Vert_{\mathcal{V}^{2}}\leq C\Big(\Vert\beta(\varphi)\Vert_{H}+\Vert\pi(\varphi)\Vert_{H} +\Vert\mu\Vert_{H}+\Vert\beta_{\Gamma}(\psi)\Vert_{H_{\Gamma}}+\Vert\pi_{\Gamma}(\psi)\Vert_{H_{\Gamma}}+\Vert\theta\Vert_{H_{\Gamma}}+\Vert\boldsymbol{\varphi}\Vert_{\mathcal{L}^{2}}\Big),\notag
\end{align}
almost everywhere on $[\eta,\infty)$. In view of \eqref{4.13}, \eqref{eq4.4}, \eqref{4.35}, and taking $\gamma$ sufficiently small in \eqref{4.35}, we find
\begin{align}
\Vert\boldsymbol{\varphi}\Vert_{\mathcal{V}^{2}}\leq \frac{1}{2}\Vert\boldsymbol{\varphi}\Vert_{\mathcal{V}^{2}}+C,\quad\text{for a.a. }s\geq\eta.\notag
\end{align}
This together with \eqref{4.13}, \eqref{eq4.4}, \eqref{4.35} yields the conclusion \eqref{4.29}.
\end{proof}
\begin{proof}[\textbf{Proof of Theorem \ref{eventual}}]
Collecting the results in Lemmas \ref{regh1}--\ref{regh3} and Remark \ref{regh1r}, we obtain the regularity properties of the weak solution $(\bm{\varphi}, \bm{\mu})$ presented in Theorem \ref{eventual}. The proof is complete.
\end{proof}
\begin{remark}\label{conti-b}\rm
From Lemmas \ref{time-regu} and \ref{regh3}, we find that $\bm{\varphi}\in C((0,\infty);\mathcal{H}^{2r})$ for any $r\in (3/4,1)$. Thanks to the Sobolev embedding theorem, for any given $\eta>0$, it holds $\varphi\in C(\overline{\Omega}\times [\eta,\infty))$ and $\psi=\varphi|_\Gamma \in C(\Gamma\times [\eta,\infty))$,  moreover,
\begin{align}
\|\varphi(t)\|_{C(\overline{\Omega})}\leq 1,\quad \|\psi(t)\|_{C(\Gamma)}\leq 1,\quad \forall\, t\geq \eta.
\label{linf1}
\end{align}
\end{remark}

\section{Separation properties of global weak solutions}
\setcounter{equation}{0}
The aim of this section is two-fold. We first prove Theorem \ref{intantaneous} on the instantaneous separation property in two dimensions, under the additional assumption $\mathbf{(A5)}$. Afterwards, we go back to Theorem \ref{eventual-2} on the eventual separation property that is valid in both two and three dimensions without $\mathbf{(A5)}$.

\subsection{Instantaneous separation in two dimensions}
Below we provide an unified approach to show the instantaneous separation property of $\bm{\varphi}$ in two dimensions, under different combination of parameters $(L,\sigma)\in[0,\infty)\times[0,\infty)$ as described in \eqref{choice}.  The proof is based on a suitable De Giorgi's iteration scheme for the balance equations of chemical potentials $\mu$ and $\theta$, that is, \eqref{eq3.5} and \eqref{eq3.6}. It extends the recent work \cite{GP} for the single Cahn-Hilliard equation subject to homogeneous Neumann boundary conditions for $\varphi$ and $\mu$. Here, we are able to obtain the separation property under possible bulk/boundary interactions for phase functions as well as chemical potentials.

\begin{proof}[\textbf{Proof of Theorem \ref{intantaneous}}]
Let $\eta>0$ be arbitrary but fixed. Thanks to Remark \ref{conti-b}, $\bm{\varphi}$ is well defined for all $t\geq \eta$ and satisfying \eqref{linf1}.
For $\delta\in(0,1)$, whose value will be chosen later, we introduce the sequence
\begin{align}
k_{n}=1-\delta-\frac{\delta}{2^{n}},\quad\forall\; n\in \mathbb{N},\notag
\end{align}
such that
\begin{align}
1-2\delta<k_{n}<k_{n+1}<1-\delta,\quad\forall\, n\geq 1,
\quad\quad k_{n}\rightarrow1-\delta\ \ \mathrm{as}\ \ n\rightarrow\infty.\notag
\end{align}
For any $n\in \mathbb{N}$, we set
\begin{align}
\varphi_{n}(x,t)=(\varphi(x,t)-k_{n})^{+},\quad\psi_{n}(x,t)=(\psi(x,t)-k_{n})^{+}.\notag
\end{align}
It is obvious that $0\leq\varphi_{n},\psi_{n}\leq 2\delta$ (cf. \cite{Po}).
Next, define
\begin{align}
A_{n}(t)=\big\{x\in\Omega:\;\varphi(x,t)-k_{n}\geq0\big\},\quad B_{n}(t)=\big\{x\in\Gamma:\;\psi(x,t)-k_{n}\geq 0\big\},\quad\forall\,t\in[\eta,\infty).\notag
\end{align}
By definition, we find
\begin{align*}
A_{n+1}(t)\subset A_{n}(t),\quad B_{n+1}(t)\subset B_{n}(t),
\qquad \forall \,n\in \mathbb{N},\quad\forall\,t\in[\eta,\infty).
\end{align*}
Consider
\begin{align*}
z_{n}(t)=\int_{A_{n}(t)}1\,\mathrm{d}x+\int_{B_{n}(t)}1\,\mathrm{d}S\geq 0,\quad\forall\, n\in \mathbb{N},\quad
\forall\,t\in[\eta,\infty).
\end{align*}
For convenience, hereafter let us consider the estimates at an arbitrary but fixed time $t\in[\eta,\infty)$ such that
\begin{align}
&\Vert\boldsymbol{\mu}(t)\Vert_{\mathcal{H}_{L,\sigma}^{1}}\leq \Vert\boldsymbol{\mu}\Vert_{L^{\infty}(\eta,\infty;\mathcal{H}_{L,\sigma}^{1})}\leq M_3,
\label{4.13b}
\\
&\Vert\beta(\varphi(t))\Vert_{L^{p}(\Omega)} + \Vert\beta(\psi(t))\Vert_{L^{p}(\Gamma)}
\leq \Vert\beta(\varphi)\Vert_{L^{\infty}(\eta,\infty;L^{p}(\Omega))} +\Vert\beta(\psi)\Vert_{L^{\infty}(\eta,\infty;L^{p}(\Gamma))}
\leq\widetilde{C}_{0}\sqrt{p},
\label{eq4.4b}
\end{align}
see \eqref{4.13} and \eqref{eq4.4}.
For any $n\in \mathbb{N}$, multiplying \eqref{eq3.5} by $\varphi_{n}$ and integrating over $\Omega$, multiplying \eqref{eq3.6} by $\psi_{n}$ and integrating over $\Gamma$, using integration by parts and then adding the resultants together, we obtain
\begin{align}
&\Vert\nabla\varphi_{n}\Vert_{H}^{2}+\Vert\nabla_{\Gamma}\psi_{n}\Vert_{H_{\Gamma}}^{2}+\underbrace{\int_{\Omega}\beta(\varphi)\varphi_{n}\,\mathrm{d}x+\int_{\Gamma}\beta_{\Gamma}(\psi)\psi_{n}\,\mathrm{d}S}_{J_{1}}\notag\\
&\quad=\underbrace{\int_{\Omega}\mu\varphi_{n}\,\mathrm{d}x+\int_{\Gamma}\theta\psi_{n}\,\mathrm{d}S}_{J_{2}} +\underbrace{\left(-\int_{\Omega}\pi(\varphi)\varphi_{n}\,\mathrm{d}x -\int_{\Gamma}\pi_{\Gamma}(\psi)\psi_{n}\,\mathrm{d}S\right)}_{J_{3}}.
\label{de4.31}
\end{align}
In the derivation of \eqref{de4.31}, we have used the following identities (cf. \cite{GP}):
\begin{align}
\int_{A_{n}(t)}\nabla\varphi\cdot\nabla\varphi_{n}\,\mathrm{d}x
=\Vert\nabla\varphi_{n}\Vert_{H}^{2},\qquad
\int_{B_{n}(t)}\nabla_{\Gamma}\psi\cdot\nabla_{\Gamma}\psi_{n}\,\mathrm{d}S =\Vert\nabla_{\Gamma}\psi_{n}\Vert_{H_{\Gamma}}^{2}.\notag
\end{align}
In what follows, we estimate the terms $J_1$--$J_3$.

\noindent \textbf{The term $J_{1}$}: As in \cite{GP}, for any $x\in A_{n}(t)$, it holds
\begin{align}
\beta\big(\varphi(x,t)\big)=\beta(k_{n})+ \int_0^1\beta'\big(s\varphi(x,t)+(1-s)k_n\big)\big(\varphi(x,t)-k_{n}\big)\,\mathrm{d}s.
\notag
\end{align}
From the assumption $\mathbf{(A1)}$ and the fact $k_{n}>1-2\delta$ we find
\begin{align}
\int_{\Omega}\beta(\varphi)\varphi_{n}\,\mathrm{d}x&=\int_{A_{n}(t)}\beta(\varphi)\varphi_{n}\,\mathrm{d}x\notag\\
&\geq\beta(k_{n})\int_{A_{n}(t)}\varphi_{n}\,\mathrm{d}x +\varpi\int_{A_{n}(t)}\varphi_{n}^{2}\,\mathrm{d}x\notag\\
&\geq\beta(1-2\delta)\int_{\Omega}\varphi_{n}\,\mathrm{d}x+\varpi\int_{\Omega}\varphi_{n}^{2}\,\mathrm{d}x.\label{de4.34}
\end{align}
In a similar manner, on the boundary it holds
\begin{align}
\int_{\Gamma}\beta_{\Gamma}(\psi)\psi_{n}\,\mathrm{d}S\geq\beta_{\Gamma}(1-2\delta)\int_{\Gamma}\psi_{n}\,\mathrm{d}S+\varpi\int_{\Gamma}\psi_{n}^{2}\,\mathrm{d}S.\label{de4.35}
\end{align}
\textbf{The term $J_{2}$}: By \eqref{4.13b}, H\"older's inequality, Agmon's inequality, Lemma \ref{critical} and the Sobolev embedding theorem, we obtain
\begin{align}
&\int_{\Omega}\mu\varphi_{n}\,\mathrm{d}x +\int_{\Gamma}\theta\psi_{n}\,\mathrm{d}S\notag\\
&\quad=\int_{A_{n}(t)}\mu\varphi_{n}\,\mathrm{d}x
+\int_{B_{n}(t)}\theta\psi_{n}\,\mathrm{d}S\notag\\
&\quad\leq\Vert\varphi_{n}\Vert_{L^{\infty}(\Omega)} \Vert\mu\Vert_{L^{p}(\Omega)}\Big(\int_{A_{n}(t)}1\,\mathrm{d}x\Big)^{1-\frac{1}{p}} +\Vert\psi_{n}\Vert_{L^{\infty}(\Gamma)}\Vert\theta\Vert_{L^{p}(\Gamma)}\Big(\int_{B_{n}(t)}1\,\mathrm{d}S\Big)^{1-\frac{1}{p}}\notag\\
&\quad\leq \left\{
\begin{array}{ll}
C\Big(\sqrt{p}\Vert\varphi_{n}\Vert_{L^{\infty}(\Omega)}\Vert\mu\Vert_{V} +|\Gamma|^{\frac{1}{p}}\Vert\psi_{n}\Vert_{L^{\infty}(\Gamma)}\Vert\theta\Vert_{V_{\Gamma}}\Big)z_{n}^{1-\frac{1}{p}}, &\quad \text{if }\sigma>0,\\
C\Big(\sqrt{p}\Vert\varphi_{n}\Vert_{L^{\infty}(\Omega)}\Vert\mu\Vert_{V} +\sqrt{p}\Vert\psi_{n}\Vert_{L^{\infty}(\Gamma)}\Vert\theta\Vert_{H^{\frac{1}{2}}(\Gamma)}\Big)z_{n}^{1-\frac{1}{p}}, &\quad \text{if }\sigma=0,
\end{array}\right.\notag\\
&\quad\leq C(\eta)\delta\sqrt{p}z_{n}^{1-\frac{1}{p}},\quad\forall\, p\geq2.\notag
\end{align}
\textbf{The term $J_{3}$}: Set $$\pi_{\infty}=\max_{s\in[-1,1]} |\pi(s)| ,\quad\pi_{\Gamma,\infty}=\max_{s\in[-1,1]}  |\pi_{\Gamma}(s)|.$$
Then it holds
\begin{align}
-\int_{\Omega}\pi(\varphi)\varphi_{n}\,\mathrm{d}x-\int_{\Gamma}\pi_{\Gamma}(\psi)\psi_{n}\,\mathrm{d}S\leq\pi_{\infty}\int_{\Omega}\varphi_{n}\,\mathrm{d}x+\pi_{\Gamma,\infty}\int_{\Gamma}\psi_{n}\,\mathrm{d}S.\label{de4.37}
\end{align}
Collecting the above estimates, we infer from \eqref{de4.31} that
\begin{align}
&\Vert\nabla\varphi_{n}\Vert_{H}^{2}+\Vert\nabla_{\Gamma}\psi_{n}\Vert_{H_{\Gamma}}^{2}+\big(\beta(1-2\delta)-\pi_{\infty}\big)\int_{\Omega}\varphi_{n}\,\mathrm{d}x+\varpi\int_{\Omega}\varphi_{n}^{2}\,\mathrm{d}x\notag\\
&\qquad+\big(\beta_{\Gamma}(1-2\delta)-\pi_{\Gamma,\infty}\big)\int_{\Gamma}\psi_{n}\,\mathrm{d}S+\varpi\int_{\Gamma}\psi_{n}^{2}\,\mathrm{d}S\notag\\
&\quad\leq C(\eta)\delta\sqrt{p}z_{n}^{1-\frac{1}{p}},\quad\forall\, p\geq2.
\label{de4.38}
\end{align}
Thanks to $\mathbf{(A1)}$, we have $$\beta(1-2\delta)-\pi_{\infty}>0,\quad\beta_{\Gamma}(1-2\delta)-\pi_{\Gamma,\infty}>0,$$
provided that $\delta\in (0,1)$ is sufficiently small.
Moreover, for any $x\in A_{n+1}(t)$, the following observation has been made in \cite{GP}:
\begin{align}
\varphi_{n}(x,t)&=\varphi(x,t)-k_n =\varphi_{n+1}(x,t)+\delta\Big(\frac{1}{2^{n}}-\frac{1}{2^{n+1}}\Big)\geq\frac{\delta}{2^{n+1}}.\label{de4.39}
\end{align}
Similarly, for any $x\in B_{n+1}(t)$, it holds
\begin{align}
\psi_{n}(x,t)\geq\frac{\delta}{2^{n+1}}.\label{de4.40}
\end{align}
From \eqref{de4.39} and \eqref{de4.40}, we see that
\begin{align}
&\int_{\Omega}|\varphi_{n}|^{3}\,\mathrm{d}x\geq\int_{A_{n+1}(t)}|\varphi_{n}|^{3}\,\mathrm{d}x\geq\Big(\frac{\delta}{2^{n+1}}\Big)^{3}\int_{A_{n+1}(t)}1\,\mathrm{d}x,\notag\\
&\int_{\Gamma}|\psi_{n}|^{3}\,\mathrm{d}S\geq\int_{B_{n+1}(t)}|\psi_{n}|^{3}\,\mathrm{d}S\geq\Big(\frac{\delta}{2^{n+1}}\Big)^{3}\int_{B_{n+1}(t)}1\,\mathrm{d}S,\notag
\end{align}
which imply
\begin{align}
\int_{\Omega}|\varphi_{n}|^{3}\,\mathrm{d}x+\int_{\Gamma}|\psi_{n}|^{3}\,\mathrm{d}S\geq\Big(\frac{\delta}{2^{n+1}}\Big)^{3}z_{n+1}.\label{de4.41}
\end{align}
On the other hand, by H\"older's inequality, the Sobolev embedding theorem and \eqref{de4.38}, we obtain
\begin{align}
\int_{\Omega}|\varphi_{n}|^{3}\,\mathrm{d}x+\int_{\Gamma}|\psi_{n}|^{3}\,\mathrm{d}S&=\int_{A_{n}(t)}|\varphi_{n}|^{3}\,\mathrm{d}x+\int_{B_{n}(t)}|\psi_{n}|^{3}\,\mathrm{d}S\notag\\
&\leq\Big(\int_{\Omega}|\varphi_{n}|^{4}\,\mathrm{d}x\Big)^{\frac{3}{4}}\Big(\int_{A_{n}(t)}1\,\mathrm{d}x\Big)^{\frac{1}{4}}+\Big(\int_{\Gamma}|\psi_{n}|^{4}\,\mathrm{d}S\Big)^{\frac{3}{4}}\Big(\int_{B_{n}(t)}1\,\mathrm{d}S\Big)^{\frac{1}{4}}\notag\\
&\leq C\Big(\Vert\varphi_{n}\Vert_{V}^{3}+\Vert\psi_{n}\Vert_{V_{\Gamma}}^{3}\Big)z_{n}^{\frac{1}{4}}\notag\\
&\leq C(\eta)\delta^{\frac{3}{2}}p^{\frac{3}{4}}z_{n}^{\frac{7}{4}-\frac{3}{2p}}.\label{de4.42}
\end{align}
 Hence, it follows from  \eqref{de4.41} and \eqref{de4.42} that
\begin{align}
z_{n+1}\leq C(\eta)2^{3n+3}\delta^{-\frac{3}{2}}p^{\frac{3}{4}}z_{n}^{\frac{7}{4}-\frac{3}{2p}}.\notag
\end{align}
Applying Lemma \ref{itera} with the choices
$$b=2^{3},\quad C=C(\eta)2^{3}\delta^{-\frac{3}{2}}p^{\frac{3}{4}}, \quad \epsilon=\frac{3}{4}-\frac{3}{2p}=\frac{3(p-2)}{4p},
$$
we can conclude that if
\begin{align}
z_{0}\leq C^{-\frac{4p}{3(p-2)}}b^{-\frac{16p^{2}}{9(p-2)^{2}}},\label{z0}
\end{align}
then $z_{n}\rightarrow 0$. The condition \eqref{z0} can be written as
\begin{align}
z_{0}\leq\underbrace{\frac{1}{2^{\frac{4p}{p-2} +\frac{16p^{2}}{3(p-2)^{2}}}C(\eta)^{\frac{4p}{3(p-2)}}}}_{=:K(\eta,p)} \left(\frac{\delta^{\frac{2p}{p-2}}}{p^{\frac{p}{p-2}}}\right).
\label{de4.44}
\end{align}
We note that the constant $K(\eta,p)$ can be independent of $p$ provided that $p>2$ is sufficiently large. \

To verify the condition \eqref{de4.44}, we can follow the same argument as in \cite{GP}. The details are presented below for the sake of completeness. Thanks to $\mathbf{(A1)}$ and the estimate \eqref{eq4.4b}, for sufficiently large $q>2$ and  sufficiently small $\delta\in (0,1)$, we find
\begin{align}
z_{0}&=\int_{A_{0}(t)}1\,\mathrm{d}x+\int_{B_{0}(t)}1\,\mathrm{d}S\notag\\
&\leq\int_{A_{0}(t)}\frac{|\beta(\varphi)|^{q}}{\beta(1-2\delta)^{q}}\,\mathrm{d}x+\int_{B_{0}(t)}\frac{|\beta(\psi)|^{q}}{\beta(1-2\delta)^{q}}\,\mathrm{d}S\notag\\
&\leq\frac{\int_{\Omega}|\beta(\varphi)|^{q}\,\mathrm{d}x}{\beta(1-2\delta)^{q}}+\frac{\int_{\Gamma}|\beta(\psi)|^{q}\,\mathrm{d}S}{\beta(1-2\delta)^{q}}\notag\\
&\leq\frac{\widetilde{C}_{0}(\eta)^{q}(\sqrt{q})^{q}}{\beta(1-2\delta)^{q}}.\notag
\end{align}
Now the additional assumption $\mathbf{(A5)}$ plays its role, namely, there exists $C_{\beta}>0$ such that
\begin{align*}
\frac{1}{\beta(1-2\delta)}\leq \frac{C_{\beta}}{|\ln \delta |^{\kappa}}\quad \text{for $\delta$ sufficiently small}.
\end{align*}
Taking $\delta=e^{-q}$ for $q>2$ sufficiently large, we get
\begin{align}
\frac{\widetilde{C}_{0}(\eta)^{q}(\sqrt{q})^{q}}{\beta(1-2\delta)^{q}} \leq\frac{\widetilde{C}_{0}(\eta)^{q}C_{\beta}^{q}(\sqrt{q})^{q}}{|\ln \delta|^{\kappa q}}=\frac{\widetilde{C}_{0}(\eta)^{q}C_{\beta}^{q}}{q^{(\kappa-1/2) q}}.\notag
\end{align}
Thus, the condition \eqref{de4.44} can be ensured if
\begin{align*}
\frac{\widetilde{C}_{0}(\eta)^{q}C_{\beta}^{q}}{q^{(\kappa-1/2) q}}\leq K(\eta,p)\frac{e^{-\frac{2qp}{p-2}}}{p^{\frac{p}{p-2}}},
\end{align*}
that is,
\begin{align}
\frac{p^{\frac{p}{p-2}}}{K(\eta,p)}\leq\widetilde{C}_{0}(\eta)^{-q} C_{\beta}^{-q}e^{-\frac{2qp}{p-2}}q^{(\kappa-1/2) q}.
\label{de4.46}
\end{align}
Fix now $p>2$. Since $\kappa>1/2$, we find
\begin{align*}
\widetilde{C}_{0}(\eta)^{-q} C_{\beta}^{-q}e^{-\frac{2qp}{p-2}}q^{(\kappa-1/2) q}=e^{-\frac{2qp}{p-2}-q\ln(\widetilde{C}_{0}(\eta)C_{\beta})+(\kappa-1/2) q\ln q}
\rightarrow\infty,\quad \text{as}\ q\rightarrow \infty.
\end{align*}
Therefore, by choosing $q$ sufficiently large (and accordingly $\delta=e^{-q}$ being small), we can ensure \eqref{de4.46}, and thus \eqref{de4.44}. This yields $z_{n}\rightarrow0$ as $n\rightarrow\infty$ for any time $t\in [\eta,\infty)$ that fulfills \eqref{4.13b} and \eqref{eq4.4b}.
Recalling the definition of $z_n$, we have
\begin{align*}
z_{n}\rightarrow|\{x\in\Omega:\;\varphi(x,t)\geq1-\delta\}|+|\{x\in\Gamma:\;\psi(x,t)\geq1-\delta\}|, \quad
\text{as}\ n\to \infty.
\end{align*}
Hence, it holds
\begin{align*}
\Vert(\varphi(t)-(1-\delta))^{+}\Vert_{L^{\infty}(\Omega)} +\Vert(\psi(t)-(1-\delta))^{+}\Vert_{L^{\infty}(\Gamma)}=0.
\end{align*}
Thanks to the continuity of $\varphi$ and $\psi$ (see Remark \ref{conti-b}), the above identity holds for all $t\geq \eta$. Besides, the inequality \eqref{de4.46} implies that $\delta$ only depends on $\eta$ but not on the specific time. Repeating exactly the same argument, we can find certain $\widetilde{\delta}\in (0,1)$ such that
\begin{align*}
\Vert(\varphi(t)-(-1+\widetilde{\delta}))^{-}\Vert_{L^{\infty}(\Omega)} +\Vert(\psi(t)-(-1+\widetilde{\delta})\big)^{-}\Vert_{L^{\infty}(\Gamma)}=0,\quad \forall\,  t\geq \eta.
\end{align*}
Taking $\delta_2=\min\{\delta,\widetilde{\delta}\}$, we arrive at the conclusion \eqref{4.31}. The proof of Theorem \ref{intantaneous} is complete.
\end{proof}

\subsection{The eventual separation property}
For any given $a\in(-1,1)$, we introduce the phase space
\begin{align*}
\mathcal{Z}_{a}=\Big\{\boldsymbol{\varphi}=(\varphi,\psi)\in \mathcal{V}^{1}:\overline{m}(\boldsymbol{\varphi})=a,E(\boldsymbol{\varphi})<\infty\Big\}.
\end{align*}
The metric $\mathrm{d}_{\mathcal{Z}_{a}}\big(\cdot,\cdot\big)$ is defined as follows:
\begin{align}
\mathrm{d}_{\mathcal{Z}_{a}}\big(\boldsymbol{\varphi}_{1},\boldsymbol{\varphi}_{2}\big)
&:=\Vert\boldsymbol{\varphi}_{1}-\boldsymbol{\varphi}_{2} \Vert_{\mathcal{V}^{1}}+\left|\int_{\Omega}\widehat{\beta}(\varphi_{1})\,\mathrm{d}x-\int_{\Omega}\widehat{\beta}(\varphi_{2})\,\mathrm{d}x\right|^{\frac{1}{2}}\notag\\
&\quad\ +\left|\int_{\Gamma}\widehat{\beta}_{\Gamma}(\psi_{1})\,\mathrm{d}S-\int_{\Gamma}\widehat{\beta}_{\Gamma}(\psi_{2})\,\mathrm{d}S\right|^{\frac{1}{2}},\quad\forall\,\boldsymbol{\varphi}_{1},\boldsymbol{\varphi}_{2}\in \mathcal{Z}_{a},\notag
\end{align}
and $\big(\mathcal{Z}_{a},\mathrm{d}_{\mathcal{Z}_{a}}\big(\cdot,\cdot\big)\big)$ is thus a completed metric space. Then we have the following conclusion, which is a straightforward extension of \cite[Proposition 4.1]{FW}:
\begin{proposition}
Assume that the assumptions in Proposition \ref{existence} are satisfied. Problem \eqref{model} defines a strongly continuous semigroup $\mathcal{S}(t):\mathcal{Z}_{\overline{m}_{0}}\rightarrow \mathcal{Z}_{\overline{m}_{0}}$ such that
\begin{align*}
\mathcal{S}(t)\boldsymbol{\varphi}_{0}=\boldsymbol{\varphi}(t),\quad \forall\, t\geq 0,
\end{align*}
where $\boldsymbol{\varphi}(t)$ is the unique global weak solution subject to the initial datum $\boldsymbol{\varphi}_{0}\in \mathcal{Z}_{\overline{m}_{0}}$.
Moreover, $\mathcal{S}(t)\in C(\mathcal{Z}_{\overline{m}_{0}}, \mathcal{Z}_{\overline{m}_{0}})$ for all $t\geq 0$.
\end{proposition}

Define $\omega$-limit set
\begin{align*}
\omega(\boldsymbol{\varphi}_{0}):=\Big\{\boldsymbol{\varphi}_{\infty}\in \mathcal{H}^{2r}\cap \mathcal{Z}_{\overline{m}_{0}}:\,\exists \,t_{n}\nearrow\infty\text{ such that  }\boldsymbol{\varphi}(t_{n})\rightarrow\boldsymbol{\varphi}_{\infty}\text{ in }\mathcal{H}^{2r}\text{ as }n\rightarrow\infty\Big\}
\end{align*}
for $r\in(3/4,1)$. Since for any $\eta>0$, $\boldsymbol{\varphi}\in L^{\infty}(\eta,\infty;\mathcal{V}^{2})$, then the trajectory $\{\mathcal{S}(t)\boldsymbol{\varphi}_{0}\}_{t\geq\eta}$ is relatively compact in $\mathcal{H}^{2r}$. As a consequence, the set $\omega(\boldsymbol{\varphi}_{0})$ is nonempty, connected and compact in $\mathcal{H}^{2r}$. Moreover, since the free energy functional $E: \mathcal{Z}_{\overline{m}_{0}}\rightarrow\mathbb{R}$ serves as a strict Lyapunov function according to Proposition \ref{energyequality}, then every $\boldsymbol{\varphi}_{\infty}\in\omega(\boldsymbol{\varphi}_{0})$ is a stationary point of $\{\mathcal{S}(t)\}_{t\geq 0}$, that is, $\mathcal{S}(t)\boldsymbol{\varphi}_{\infty}=\boldsymbol{\varphi}_{\infty}$ for all $t\geq 0$. Hence, $(\boldsymbol{\varphi}_{\infty},\boldsymbol{\mu}_{\infty})$ can be regarded as the global weak solution of the following problem
\begin{align}
\left\{
\begin{array}{ll}
\partial_{t}\varphi_{\infty}=\Delta\mu_{\infty},&\text{in }Q,\\
\mu_{\infty}=-\Delta \varphi_{\infty}+\beta(\varphi_{\infty})+\pi(\varphi_{\infty}),&\text{in }Q,\\
\partial_{t}\psi_{\infty}=\sigma\Delta_{\Gamma}\theta_{\infty} -\partial_{\mathbf{n}}\mu_{\infty},&\text{on }\Sigma,\\
\theta_{\infty}=-\Delta_{\Gamma}\psi_{\infty}+\beta_{\Gamma}(\psi_{\infty})+\pi_{\Gamma}(\psi_{\infty}) +\partial_{\mathbf{n}}\varphi_{\infty},&\text{on }\Sigma,\\
\varphi_{\infty}|_{\Gamma}=\psi_{\infty},&\text{on }\Sigma,\\
L\partial_{\mathbf{n}}\mu_{\infty}=\theta_{\infty}-\mu_{\infty}|_{\Gamma},&\text{on }\Sigma,\\
(\varphi_{\infty},\psi_{\infty})|_{t=0}=(\varphi_{\infty},\psi_{\infty}),\quad\text{with }\varphi_{\infty}|_{\Gamma}=\psi_{\infty},&\text{in }\Omega\times\Gamma.
\end{array}\right.\notag
\end{align}
Thanks to the regularity of weak solutions for the evolutionary problem (recall Lemmas \ref{regh2}, \ref{regh3} and Remark \ref{regh1r}), we obtain
\begin{align*}
\boldsymbol{\varphi}_{\infty}\in\mathcal{V}^{2}, \quad\boldsymbol{\mu}_{\infty}\in\mathcal{V}_{\sigma}^{2}, \quad(\beta(\varphi_{\infty}),\beta_{\Gamma}(\psi_{\infty}))\in\mathcal{L}^{2}.
\end{align*}
Hence, $(\boldsymbol{\varphi}_{\infty},\boldsymbol{\mu}_{\infty})$ is a strong solution of the stationary problem
\begin{align}
\left\{
\begin{array}{ll}
\Delta\mu_{\infty}=0,&\text{in }\Omega,\\
\mu_{\infty}=-\Delta \varphi_{\infty}+\beta(\varphi_{\infty})+\pi(\varphi_{\infty}),&\text{in }\Omega,\\
\sigma\Delta_{\Gamma}\theta_{\infty}-\partial_{\mathbf{n}}\mu_{\infty}=0, &\text{on }\Gamma,\\
\theta_{\infty}=-\Delta_{\Gamma}\psi_{\infty}+\beta_{\Gamma}(\psi_{\infty}) +\pi_{\Gamma}(\psi_{\infty})+\partial_{\mathbf{n}}\varphi_{\infty},&\text{on }\Gamma,\\
\varphi_{\infty}|_{\Gamma}=\psi_{\infty},&\text{on }\Gamma,\\
L\partial_{\mathbf{n}}\mu_{\infty}=\theta_{\infty}-\mu_{\infty}|_{\Gamma},&\text{on }\Gamma.
\end{array}\right.\label{eq4.7}
\end{align}
Testing \eqref{eq4.7}$_{1}$ by $\mu_{\infty}$ and \eqref{eq4.7}$_{3}$ by $\theta_{\infty}$, respectively,
adding the resultants and using
\eqref{eq4.7}$_{6}$, we obtain
\begin{align}
\int_{\Omega}|\nabla\mu_{\infty}|^{2}\,\mathrm{d}x +\sigma\int_{\Gamma}|\nabla_{\Gamma}\theta_{\infty}|^{2}\,\mathrm{d}S +L\int_{\Gamma}|\partial_{\mathbf{n}}\mu_{\infty}|^{2}\,\mathrm{d}S=0.\notag
\end{align}
For parameters $(L,\sigma)$ satisfying \eqref{choice}, we can conclude that $\mu_{\infty}=\theta_{\infty}$ are constants in all cases. Furthermore, integrating \eqref{eq4.7}$_{2}$ in $\Omega$, and \eqref{eq4.7}$_{4}$ on $\Gamma$, we get
\begin{align}
\mu_{\infty}=\theta_{\infty}=\frac{1}{|\Omega|+|\Gamma|} \Big[\int_{\Omega}\big(\beta(\varphi_{\infty})+\pi(\varphi_{\infty})\big)\,\mathrm{d}x +\int_{\Gamma}\big(\beta_{\Gamma}(\psi_{\infty})+\pi_{\Gamma}(\psi_{\infty})\big)\,\mathrm{d}S\Big].
\label{eq4.10}
\end{align}
Thus, the stationary problem \eqref{eq4.7} reduces to
\begin{align}
\left\{
\begin{array}{ll}
\mu_{\infty}=-\Delta \varphi_{\infty}+\beta(\varphi_{\infty})+\pi(\varphi_{\infty}),&\text{ in }\Omega,\\
\theta_{\infty}=-\Delta_{\Gamma}\psi_{\infty}+\beta_{\Gamma}(\psi_{\infty}) +\pi_{\Gamma}(\psi_{\infty})+\partial_{\mathbf{n}}\varphi_{\infty},&\text{ on }\Gamma,\\
\varphi_{\infty}|_{\Gamma}=\psi_{\infty},&\text{ on }\Gamma,
\end{array}\right.\label{stationary}
\end{align}
with $\mu_{\infty}=\theta_{\infty}$ given by \eqref{eq4.10}. Finally, due to Corollary \ref{decay1}, we see that  $E(\bm{\varphi}_\infty)\equiv E_\infty$ on the $\omega$-limit set $\omega(\boldsymbol{\varphi}_{0})$.

\begin{remark}\label{sta-ob}\rm
For different combination of parameters $(L,\sigma)$ satisfying \eqref{choice}, the chemical potentials $\mu$, $\theta$ in the evolutionary problem present different relationships and regularity properties in different spaces. However, for all cases, the stationary problem satisfies the same form \eqref{stationary}.
\end{remark}

Below we present two approaches that yield the eventual separation of global weak solutions.

\begin{proof}[\textbf{Proof of Theorem \ref{eventual-2}: a dynamic approach}] The eventual separation property can obtained by a classical dynamical approach, see, e.g., \cite{AW,GMS09,FW}.
 In particular, Remark \ref{sta-ob} enables us to take advantage of known results on the stationary problem obtained in \cite{FW}.
Similar to \cite[Definition 4.1]{FW}, given $a\in(-1,1)$, we say $\boldsymbol{\varphi}_{\mathrm{s}}=(\varphi_{\mathrm{s}},\psi_{\mathrm{s}})$ is a steady state of problem \eqref{model}, if $\boldsymbol{\varphi}_{\mathrm{s}}\in \mathcal{Z}_{a}$, $(\beta(\varphi_{\mathrm{s}}),\beta_{\Gamma}(\psi_{\mathrm{s}}))\in \mathcal{L}^2$ and
\begin{align}
&\int_{\Omega}\nabla \varphi_{\mathrm{s}}\cdot\nabla z\,\mathrm{d}x+\int_{\Omega}\Big(\beta(\varphi_{\mathrm{s}}) +\pi(\varphi_{\mathrm{s}})-\mu_{\mathrm{s}}\Big)z\,\mathrm{d}x +\int_{\Gamma}\nabla_{\Gamma}\psi_{\mathrm{s}}\cdot\nabla_{\Gamma}z_{\Gamma}\,\mathrm{d}S\notag\\
&\quad+\int_{\Gamma}\Big(\beta_{\Gamma}(\psi_{\mathrm{s}}) +\pi_{\Gamma}(\psi_{\mathrm{s}})-\theta_{\mathrm{s}}\Big)z_{\Gamma}\,\mathrm{d}S=0, \quad\forall\,\boldsymbol{z}=(z,z_{\Gamma})\in\mathcal{V}^{1},\notag
\end{align}
with $\mu_{\mathrm{s}}=\theta_{\mathrm{s}}$ given as in \eqref{eq4.10}. Denote the set of all steady states by $\boldsymbol{\mathcal{S}}_{a}$. It has been shown in \cite[Lemma 4.1]{FW} that under the assumptions  $\mathbf{(A1)}$--$\mathbf{(A3)}$, for any given $a\in(-1,1)$, there exist uniform constants $M_{a}>0$ and $\delta_{a}\in(0,1)$ such that for every $\boldsymbol{\varphi}_{\mathrm{s}}=(\varphi_{\mathrm{s}},\psi_{\mathrm{s}})\in\boldsymbol{\mathcal{S}}_{a}$ and its associated constants $\mu_{\mathrm{s}}=\theta_{\mathrm{s}}$, it holds
\begin{align}
&-1+\delta_{a}\leq \varphi_{\mathrm{s}}\leq1-\delta_{a},\quad\text{in }\Omega,\notag\\
&-1+\delta_{a}\leq \psi_{\mathrm{s}}\leq1-\delta_{a},\quad\text{on }\Gamma,\notag\\
&|\mu_{\mathrm{s}}|=|\theta_{\mathrm{s}}|\leq M_{a}.\notag
\end{align}
Moreover, $\boldsymbol{\mathcal{S}}_{a}$ is a bounded set in $\mathcal{H}^{3}$.
Returning to our problem \eqref{model} with $a=\overline{m}_0$, we have already shown  $$\omega(\boldsymbol{\varphi}_{0})\subset\boldsymbol{\mathcal{S}}_{\overline{m}_{0}}.$$
Then it follows from the definition of $\omega(\boldsymbol{\varphi}_{0})$ that
\begin{align*}
\lim_{t\rightarrow\infty}\text{dist}\big(S(t)\boldsymbol{\varphi}_{0},\omega(\boldsymbol{\varphi}_{0})\big)=0\quad\text{in }\mathcal{H}^{2r},\quad r\in (3/4,1),
\end{align*}
where the distance is given by $\text{dist}\big(\boldsymbol{z},\omega(\boldsymbol{\varphi}_{0})\big)=\inf_{\boldsymbol{y}\in\omega(\boldsymbol{\varphi}_{0})}\Vert\boldsymbol{z}-\boldsymbol{y}\Vert_{\mathcal{H}^{2r}}$. Thanks to the Sobolev embedding theorem, we see that $\mathcal{H}^{2r}\hookrightarrow C(\overline{\Omega})\times C(\Gamma)$ for $r\in(3/4,1)$. Therefore, we can conclude \eqref{eq4.13} by taking $$\delta_{1}=\frac12\delta_{\overline{m}_{0}}\in (0,1).$$
This completes the proof of  Theorem \ref{eventual-2}.
\end{proof}

\begin{proof}[\textbf{Proof of Theorem \ref{eventual-2}: via De Giorgi's iteration}]
Inspired by the recent work \cite{GP}, we can give an alternative proof based on the dissipative nature of problem \eqref{model} and a suitable De Giorgi's iteration scheme.
Denote the different quotient of $f$ by
$\partial_{t}^{h}f(\cdot)=\big(f(\cdot+h)-f(\cdot)\big)/h$ with $h>0$.
Similar to the derivation of \eqref{diff1}, we obtain
\begin{align}
	\frac{\mathrm{d}}{\mathrm{d}t}\Vert\partial_{t}^{h}\boldsymbol{\varphi}(t)\Vert_{\mathcal{H}_{L,\sigma,0}^{-1}}^{2}
	+\Vert\partial_{t}^{h}\boldsymbol{\varphi}(t)\Vert_{\mathcal{V}_{(0)}^{1}}^{2}\leq C\Vert\partial_{t}^{h}\boldsymbol{\varphi}(t)\Vert_{\mathcal{H}_{L,\sigma,0}^{-1}}^{2}.
\label{quot3}
\end{align}
Recalling Corollary \ref{decay1}, for any $\widetilde{\epsilon}>0$, there exists some $T_{\widetilde{\epsilon}}\geq 1$, such that
\begin{align*}
	\int_{t}^{t+2}\Vert\boldsymbol{\varphi}'(s)\Vert_{\mathcal{H}_{L,\sigma,0}^{-1}}^{2}\,\mathrm{d}s
	<\widetilde{\epsilon},\quad\forall\,t\geq T_{\widetilde{\epsilon}}.
\end{align*}
Then for $h\in (0,1)$ sufficiently small, it holds
\begin{align}
\int_{t}^{t+1}\Vert\partial_{t}^{h}\boldsymbol{\varphi}(s)\Vert_{\mathcal{H}_{L,\sigma,0}^{-1}}^{2}\,\mathrm{d}s
	\leq C\int_{t}^{t+2}\Vert\boldsymbol{\varphi}'(s)\Vert_{\mathcal{H}_{L,\sigma,0}^{-1}}^{2}\, \mathrm{d}s<C\widetilde{\epsilon},
	\quad\forall\,t\geq T_{\widetilde{\epsilon}}.\label{quot5}
\end{align}
By \eqref{quot3}, \eqref{quot5} and the well-known uniform Gronwall inequality,
we infer that
\begin{align}
	\sup_{t\geq T_{\widetilde{\epsilon}}+1}\Vert\partial_{t}^{h}\boldsymbol{\varphi}(t)\Vert_{\mathcal{H}_{L,\sigma,0}^{-1}}^{2} \leq C\widetilde{\epsilon},\notag
\end{align}
for all $h\in(0,1)$ sufficiently small, where $C>0$ is independent of $h$. Passing to the limit as $h\to 0$, we get
\begin{align}
	\sup_{t\geq T_{\widetilde{\epsilon}}+1}\Vert\boldsymbol{\varphi}'(t)\Vert_{\mathcal{H}_{L,\sigma,0}^{-1}}^{2}
	\leq C\widetilde{\epsilon},\notag
\end{align}
which, combined with \eqref{defn2.1}, yields
 \begin{align}
 	\sup_{t\geq T_{\widetilde{\epsilon}}+1}\Vert\mathbf{P}\boldsymbol{\mu}(t)\Vert_{\mathcal{H}_{L,\sigma,0}^{1}}^{2}\leq C\widetilde{\epsilon}.\label{quot6}
 \end{align}
Let us now return to the De Giorgi's iteration scheme in the proof of Theorem \ref{intantaneous}.
Apply the decomposition as in \cite{GP}:
 \begin{align}
 	\int_{\Omega}\mu\varphi_{n}\,\mathrm{d}x+\int_{\Gamma}\theta\psi_{n}\,\mathrm{d}S &=\int_{\Omega}(\mu-\overline{m}(\boldsymbol{\mu}))\varphi_{n}\,\mathrm{d}x
 	+\int_{\Gamma}(\theta-\overline{m}(\boldsymbol{\mu}))\psi_{n}\,\mathrm{d}S\notag\\
 	&\quad+\overline{m}(\boldsymbol{\mu})\Big(\int_{\Omega}\varphi_{n}\,\mathrm{d}x
 	+\int_{\Gamma}\psi_{n}\,\mathrm{d}S\Big).
 \label{quot7}
 \end{align}
Recalling \eqref{de4.31}, \eqref{de4.34}, \eqref{de4.35}, \eqref{de4.37}, and the facts $0\leq\varphi_{n},\psi_{n}\leq2\delta$, we infer from \eqref{quot7}, the Sobolev embedding theorem and the generalized Poincar\'{e}'s inequalities (see Lemmas \ref{equivalent}, \ref{equivalentb}) that
\begin{align}
&\Vert\nabla\varphi_{n}\Vert_{H}^{2}+\Vert\nabla_{\Gamma}\psi_{n}\Vert_{H_{\Gamma}}^{2}
+\big(\beta(1-2\delta)-\pi_{\infty}-\overline{m}(\boldsymbol{\mu})\big)\int_{\Omega}\varphi_{n}\,\mathrm{d}x
+\varpi\int_{\Omega}\varphi_{n}^{2}\,\mathrm{d}x\notag\\
&\qquad+\big(\beta_{\Gamma}(1-2\delta)-\pi_{\Gamma,\infty}
-\overline{m}(\boldsymbol{\mu})\big)\int_{\Gamma}\psi_{n}\,\mathrm{d}S
+\varpi\int_{\Gamma}\psi_{n}^{2}\,\mathrm{d}S\notag\\
&\quad=\int_{\Omega}(\mu-\overline{m}(\boldsymbol{\mu}))\varphi_{n}\,\mathrm{d}x
+\int_{\Gamma}(\theta-\overline{m}(\boldsymbol{\mu}))\psi_{n}\,\mathrm{d}S\notag\\
&\quad\leq C\delta\Big(\Vert\mu-\overline{m}(\boldsymbol{\mu})\Vert_{L^{4}(\Omega)}
+\Vert\theta-\overline{m}(\boldsymbol{\mu})\Vert_{L^{4}(\Gamma)}\Big)z_{n}^{\frac{3}{4}}\notag\\
&\quad\leq C\delta\Vert\mathbf{P}\boldsymbol{\mu}\Vert_{\mathcal{H}_{L,\sigma,0}^{1}}z_{n}^{\frac{3}{4}}.
\label{quot8}
\end{align}
We emphasize that the estimate \eqref{quot8} is valid for all cases in \eqref{choice}.
Thanks to Lemma \ref{regh1}, it holds $\overline{m}(\boldsymbol{\mu})\in L^{\infty}(1,\infty)$. Then
by $\mathbf{(A1)}$, we can take $\delta>0$ sufficiently small such that
\begin{align}
	\beta(1-2\delta)-\pi_{\infty}-\overline{m}(\boldsymbol{\mu}(t))>0,
	\quad\beta_{\Gamma}(1-2\delta)-\pi_{\Gamma,\infty}-\overline{m}(\boldsymbol{\mu}(t))>0,
	\quad\text{for a.a.}\ t\geq 1.\label{quot9}
\end{align}
Combining \eqref{quot6}, \eqref{quot8} and \eqref{quot9}, we find
\begin{align}
\Vert\nabla\varphi_{n}\Vert_{H}^{2}
+\Vert\nabla_{\Gamma}\psi_{n}\Vert_{H_{\Gamma}}^{2}
+\varpi\Big(\int_{\Omega}\varphi_{n}^{2}\,\mathrm{d}x
+\int_{\Gamma}\psi_{n}^{2}\,\mathrm{d}S\Big)
\leq C\widetilde{\epsilon}^{\frac{1}{2}}\delta z_{n}^{\frac{3}{4}},\quad\text{for a.a.}\ t\geq T_{\widetilde{\epsilon}}+1,\notag
\end{align}
which yields
$$
z_{n+1}\leq C2^{3n+3} \delta^{-\frac{3}{2}}\widetilde{\epsilon}^{\frac{3}{4}}z_{n}^{\frac{11}{8}}.
$$
Hence, for almost all $t\geq T_{\widetilde{\epsilon}}+1$, we can conclude that $z_{n}(t)\rightarrow0$ as $n\rightarrow\infty$,
provided that
 \begin{align}
 	z_{0}(t)\leq\frac{1}{2^\frac{88}{3}C^\frac{8}{3}}\left(\frac{\delta^{4}}{\widetilde{\epsilon}^2}\right).
 \label{quot10}
 \end{align}
In view of \eqref{quot10} and the fact $z_{0}(t)\leq |\Omega|+|\Gamma|$, for the previously chosen $\delta$, we can take $\widetilde{\epsilon}>0$ sufficiently small such that
 \begin{align*}
 |\Omega|+|\Gamma|\leq\frac{1}{2^\frac{88}{3}C^\frac{8}{3}}\left(\frac{\delta^{4}}{\widetilde{\epsilon}^2}\right).
 \end{align*}
Then for the corresponding time $T_{\widetilde{\epsilon}}$, \eqref{quot10} holds for almost all $t\geq T_{\widetilde{\epsilon}}+1$.
This gives
\begin{align*}
\Vert(\varphi(t)-(1-\delta))^{+}\Vert_{L^{\infty}(\Omega)} +\Vert(\psi(t)-(1-\delta))^{+}\Vert_{L^{\infty}(\Gamma)}=0,\quad \text{for a.a.}\ t\geq T_{\widetilde{\epsilon}}+1.
\end{align*}
The rest of the proof is similar to that for Theorem \ref{intantaneous}.
\end{proof}

\begin{remark}\rm
The smallness condition \eqref{quot6} plays a crucial role in the above proof. This property can be guaranteed provided that the free energy cannot ``drop to much'' during the evolution. In view of the energy equality \eqref{Energyequ}, it happens either for sufficiently large time (as studied here), or instantaneously if the initial datum is close to an absolute minimizer of the total free energy (this is the situation of ``small initial energy'' mentioned in \cite{GP} for the Cahn-Hilliard equation subject to homogeneous Neumann boundary conditions). Another interesting case is that the initial datum is close to a local minimizer of the free energy, which could be handled by the {\L}ojasiewicz-Simon approach (see \cite{LW}), under the additional assumption that the nonlinearities $\beta$, $\beta_\Gamma$, $\pi$ and $\pi_\Gamma$ are real analytic.
\end{remark}

\section{Long-time behavior}
\setcounter{equation}{0}
In this section, we prove Theorem \ref{equilibrium} on the long-time behavior of global weak solutions to problem \eqref{model}.

\begin{lemma}
Let $\boldsymbol{\varphi}$ be the unique global weak solution to problem \eqref{model} obtained in Proposition \ref{existence}. There exists a positive constant $M_{5}$ such that
\begin{align}
\Vert\boldsymbol{\varphi}\Vert_{L^{\infty}(T_{\mathrm{SP}},\infty;\mathcal{H}^{2+r})}\leq M_{5},\label{6.1}
\end{align}
where $r\in (0,1/2)$ and $T_{\mathrm{SP}}>0$ is determined as in Theorem \ref{eventual-2}.
\end{lemma}
\begin{proof}
Consider the elliptic problem for $\boldsymbol{\varphi}=(\varphi,\psi)$:
\begin{align}
&-\Delta \varphi(t)=\mu(t)-\beta\big(\varphi(t)\big)-\pi\big(\varphi(t)\big)=:h(t),&&\text{a.e. in }\Omega,\notag\\
&\varphi|_{\Gamma}(t)=\psi(t),&&\text{a.e. on }\Gamma,\notag\\
&\partial_{\mathbf{n}}\varphi(t)-\Delta_{\Gamma}\psi(t)+\psi(t)=\theta(t) -\beta_{\Gamma}\big(\psi(t)\big)-\pi_{\Gamma}\big(\psi(t)\big)+\psi(t):=h_{\Gamma}(t),&&\text{a.e. on }\Gamma,\notag
\end{align}
for almost all $t\geq T_{\mathrm{SP}}$. From the strict separation property \eqref{eq4.13}, \eqref{uni1} and $\mathbf{(A1)}$, we obtain
\begin{align*}
\Vert\beta\big(\varphi(t)\big)\Vert_{V}+\Vert\beta_{\Gamma}\big(\psi(t)\big)\Vert_{V_{\Gamma}}\leq C,\quad\text{for a.a. }t\geq T_{\mathrm{SP}}.
\end{align*}
Next, it follows from \eqref{4.13} that
\begin{align*}
& \Vert\mu(t)\Vert_{V}+\Vert\theta(t)\Vert_{V_{\Gamma}}\leq C,\qquad\quad  \text{for a.a. }t\geq T_{\mathrm{SP}},\quad \text{if }\sigma>0,\\
& \Vert\mu(t)\Vert_{V}+\Vert\theta(t)\Vert_{H^{1/2}(\Gamma)}\leq C,\quad \text{for a.a. }t\geq T_{\mathrm{SP}},\quad \text{if }\sigma=0.
\end{align*}
As a consequence, we have
\begin{align*}
&\Vert h(t)\Vert_{V}+\Vert h_{\Gamma}(t)\Vert_{V_{\Gamma}}\leq C,\qquad\quad \text{for a.a. }t\geq T_{\mathrm{SP}},\quad \text{if }\sigma>0,\\
& \Vert h(t)\Vert_{V}+\Vert h_{\Gamma}(t)\Vert_{H^{1/2}(\Gamma)}\leq C,\quad\text{for a.a. }t\geq T_{\mathrm{SP}},\quad \text{if }\sigma=0.
\end{align*}
Applying the elliptic regularity theorem (see \cite[Corollary A.1]{GMS09}), we get the uniform estimate \eqref{6.1}.
\end{proof}

By the estimates \eqref{4.6}, \eqref{6.1} and Lemma \ref{ALS}, we find that $\boldsymbol{\varphi}\in C([t,t+1];\mathcal{V}^{2})$, for all $ t\geq T_{\mathrm{SP}}$. Hence, it holds
\begin{align*}
\boldsymbol{\varphi}\in C([T_{\mathrm{SP}},\infty);\mathcal{V}^{2}).
\end{align*}
The uniform estimate \eqref{6.1} and the compact embedding $\mathcal{H}^{2+r}\hookrightarrow\hookrightarrow \mathcal{H}^{2}$ imply that the $\omega$-limit set $\omega(\boldsymbol{\varphi}_{0})$ is non-empty and compact in $\mathcal{V}^{2}$. As a result,
\begin{align*}
\lim_{t\rightarrow\infty}\text{dist}\big(\mathcal{S}(t)\boldsymbol{\varphi}_{0},\omega(\boldsymbol{\varphi}_{0})\big)=0\quad\text{in }\mathcal{V}^{2}.
\end{align*}
Finally, to prove that the $\omega$-limit set $\omega(\bm{\varphi}_0)$ reduces to a singleton, we apply the {\L}ojasiewicz-Simon approach, see, e.g., for applications to the Cahn-Hilliard type equations \cite{AW,GKY,GGM,FW,LW,Wu}. For our current setting, the main tool is the following extended {\L}ojasiewicz-Simon inequality (see \cite[Lemma 5.2]{FW}). Given $a\in(-1,1)$, it is straightforward to verify that every $\boldsymbol{\zeta}=(\zeta,\zeta_{\Gamma})\in\boldsymbol{\mathcal{S}}_{a}$ is a critical point of the free energy $E$ defined by \eqref{energy}. Moreover, we have
\begin{lemma}\label{LS}
Suppose that $\mathbf{(A1)}$--$\mathbf{(A3)}$ hold. In addition, we assume that $\beta$, $\beta_{\Gamma}$ are real analytic on $(-1,1)$ and $\pi$, $\pi_{\Gamma}$ are real analytic on $\mathbb{R}$. Let $\boldsymbol{\zeta}=(\zeta,\zeta_{\Gamma})\in\boldsymbol{\mathcal{S}}_{a}$ for some $a\in(-1,1)$. There exist constants $\theta^{*}\in(0,1/2)$ and $b^{*}>0$ such that
\begin{align}
\Big\Vert
\mathbf{P}
\Big(
\begin{array}{c}
-\Delta \varphi+\beta(\varphi)+\pi(\varphi)\\
\partial_{\mathbf{n}}\varphi-\Delta_{\Gamma}\psi+\beta_{\Gamma}(\psi)+\pi_{\Gamma}(\psi)
\end{array}\Big)\Big\Vert_{\mathcal{L}^{2}} \geq\big|E\big(\boldsymbol{\varphi}\big)-E\big(\boldsymbol{\zeta}\big)\big|^{1-\theta^{*}},
\label{6.2}
\end{align}
for all $\boldsymbol{\varphi}=(\varphi, \psi)\in\mathcal{V}^{2}$ satisfying $\Vert\boldsymbol{\varphi}-\boldsymbol{\zeta}\Vert_{\mathcal{V}^{2}}<b^{*}$ and $\overline{m}(\boldsymbol{\varphi})=a$.
\end{lemma}
\begin{proof}[\textbf{Proof of Theorem \ref{equilibrium}}]
With the aid of Lemma \ref{LS}, the proof can be carried out in a standard procedure, we refer to \cite{FW} for the specific case with $L=0$, $\sigma\in [0,\infty)$, see also \cite{GW} for the case $L=0$, $\sigma=0$ and \cite{GKY} for the case $L, \sigma\in (0,\infty)$, both with regular potentials. The detailed proof will not be repeated here and we just mention some key points. First, we can show that the unique global weak solution $\boldsymbol{\varphi}$ of problem \eqref{model} will fall into a small $\mathcal{V}^{2}$-neighborhood of $\omega(\boldsymbol{\varphi}_{0})$ from a sufficiently large time on. Then, by the generalized Poincar\'e's inequalities (see Lemmas \ref{equivalent}, \ref{equivalentb}) and Corollary \ref{decay1}, we find
\begin{align*}
\Vert\boldsymbol{\varphi}'\Vert_{\mathcal{H}_{L,\sigma,0}^{-1}}=\Vert\mathbf{P}\boldsymbol{\mu}\Vert_{\mathcal{H}_{L,\sigma,0}^{1}}\geq C|E\big(\boldsymbol{\varphi}\big)-E_\infty|^{1-\theta^{\ast}}.
\end{align*}
This inequality combined with the energy equality \eqref{Energyequ} and the {\L}ojasiewicz-Simon inequality \eqref{6.2} yields integrability of the time derivative of $\boldsymbol{\varphi}$ such that
$\boldsymbol{\varphi}'\in L^1(T_{\mathrm{SP}},\infty; \mathcal{H}_{L,\sigma,0}^{-1})$.
Consequently, the convergence $\boldsymbol{\varphi}(t)\rightarrow\boldsymbol{\varphi}_{\infty}$ holds in $\mathcal{H}_{L,\sigma,0}^{-1}$ as $t\rightarrow\infty$. With the compactness of the trajectory
of $\boldsymbol{\varphi}$ (see \eqref{6.1}), we can further conclude $\boldsymbol{\varphi}(t)\rightarrow\boldsymbol{\varphi}_{\infty}$ in $\mathcal{V}^{2}$ as $t\rightarrow\infty$. Finally, the convergence rate \eqref{converrate} follows from an argument similar to that in \cite{GW}.
\end{proof}

\appendix
\section{Useful tools}
\setcounter{equation}{0}
\noindent We report some technical lemmas that have been used
in our analysis. First, we recall the compactness lemma of Aubin-Lions-Simon
type (see, for instance, \cite{Lions} in the case $q>1$ and \cite{Simon}
when $q=1 $).

\begin{lemma}
\label{ALS} Let $X_{0} \overset{c}{\hookrightarrow } X_{1}\subset X_{2}$
where $X_{j}$ are (real) Banach spaces ($j=0,1,2$). Let $1<p\leq \infty $, $%
1\leq q\leq \infty ~$and $I$ be a bounded subinterval of $\mathbb{R}$. Then,
the sets
\begin{equation*}
\left\{ \varphi \in L^{p}\left( I;X_{0}\right) :\partial _{t}\varphi \in
L^{q}\left( I;X_{2}\right) \right\} \overset{c}{\hookrightarrow }
L^{p}\left( I;X_{1}\right),\quad \text{ if }1<p<\infty,
\end{equation*}
and
\begin{equation*}
\left\{ \varphi \in L^{p}\left( I;X_{0}\right) :\partial _{t}\varphi \in
L^{q}\left( I;X_{2}\right) \right\} \overset{c}{\hookrightarrow } C\left(
I;X_{1}\right),\quad \text{ if }p=\infty ,\text{ }q>1.
\end{equation*}
\end{lemma}
The proof of the following Ehrling lemma can be found in \cite{Lions}.
\begin{lemma}
\label{Ehrling}
Let $B_{0}$, $B_{1}$, $B$ be three Banach spaces so that $B_{0}$ and $B_{1}$ are reflexive. Moreover, $B_{0}\hookrightarrow\hookrightarrow B\hookrightarrow B_{1}$. Then, for each $\epsilon>0$, there exists a positive constant $C_{\epsilon}$ depends on $\epsilon$ such that$$\Vert z\Vert_{B}\leq\epsilon\Vert z\Vert_{B_{0}}+C_{\epsilon}\Vert z\Vert_{B_{1}},\quad \text{for all }z\in B_{0}.$$
\end{lemma}

The following generalized Poincar\'{e} type inequality has been proved in \cite[Lemma A.1, Corollary A.2]{KL} (for $\sigma=1$, the extension to $\sigma\in (0,\infty)$ is straightforward, see also \cite{CF15}):
\begin{lemma} \label{equivalent}
  There exists a constant $c_P>0$ depending only on $L\in [0,\infty)$, $\sigma\in (0,\infty)$ and $\Omega$ such that
 \begin{align}
  \|(y,y_{\Gamma})\|_{\mathcal{L}^2}\leq c_P \Vert(y,y_{\Gamma})\Vert_{\mathcal{H}^{1}_{L,\sigma,0}},\quad \forall\, (y,y_{\Gamma})\in \mathcal{H}^{1}_{L,\sigma,0}.
  \notag
  \end{align}
\end{lemma}
We also recall the generalized Poincar\'{e}'s inequality proved in \cite[Section 2]{Gal}:
\begin{lemma}\label{equivalentb}
  There exists a constant $\widetilde{c}_P>0$ depending only on $\Omega$ such that
  \begin{align}
   \Vert y\Vert_{H}\leq \widetilde{c}_{P} \Vert\nabla y\Vert_{H},\quad\forall \,(y,y_{\Gamma})\in \widetilde{\mathcal{V}}_{(0)}^1.
   \notag
  \end{align}
\end{lemma}
\noindent

The following lemma is useful when we apply De Giorgi's iteration (cf. \cite{Di,Po}).
\begin{lemma}
\label{itera}
Let $\{z_{n}\}_{n\in\mathbb{N}}\subset \mathbb{R}^{+}$ satisfy the recursive inequalities:
\begin{align}
z_{n+1}\leq Cb^{n}z_{n}^{1+\epsilon},\quad\forall \;n\geq0,\notag
\end{align}
for some constants $C>0$, $b>1$, and $\epsilon>0$. If $z_{0}\leq\varsigma:= C^{-1/\epsilon}b^{-1/\epsilon^{2}}$, then
\begin{align}
z_{n}\leq \varsigma b^{-n/\epsilon},\quad\forall \;n\geq0,\notag
\end{align}
and consequently, $z_{n}\rightarrow0$ as $n\rightarrow\infty.$
\end{lemma}

The following critical Sobolev estimate can be proved by the definition of Sobolev spaces on smooth compact manifold and the corresponding critical estimate in whole space (cf. \cite{BCD})

\begin{lemma}
	\label{critical}
	Let $\mathcal{M}$ be a smooth compact $d$-dimension manifold, then there holds
	\begin{align*}
	\Vert f\Vert_{L^{p}(\mathcal{M})}\leq C\sqrt{p}\Vert f\Vert_{H^{\frac{d}{2}}(\mathcal{M})},\quad\forall\,f\in H^{\frac{d}{2}}(\mathcal{M}),
	\end{align*}
	where the constant $C>0$ is independent of $p\geq2$.
\end{lemma}

\medskip

\noindent \textbf{Acknowledgments.}
H. Wu is a member of the Key Laboratory of Mathematics for Nonlinear Sciences (Fudan University), Ministry
of Education of China. The research of H. Wu was partially supported by NNSFC Grant No. 12071084
and the Shanghai Center for Mathematical Sciences at Fudan University.

%

\end{document}